\documentclass[twoside]{amsart}
\usepackage{latexsym}
\usepackage{amssymb,amsmath,amsopn}
\usepackage[dvips]{graphicx}   %To insert ps figures
               {\begin{list}{}{\leftmargin#1\rightmargin#2}\item{}}%
               {\end{list}}

\newtheorem{thm}{Theorem}{\bf}{\sl}
\newtheorem{lem}{Lemma}{\bf}{\sl}
\newtheorem{defn}{Definition}
\newtheorem{rem}{Remark}{\bf}{\rm}
\newtheorem{con}{Conjecture}{\bf}{\rm}
\renewcommand{\Re}{\mathbb R}
\renewcommand{\S}{\mathbb S}
\newcommand{\B}{\mathbf B}

\newcommand{\dif}{\;\mathrm{d}}
\def\bea{\begin{eqnarray}}
\def\eea{\end{eqnarray}}
\def\ben{\begin{equation}}
\def\een{\end{equation}}

\DeclareMathOperator{\inter}{int}
\DeclareMathOperator{\bd}{bd}

\DeclareMathOperator{\grad}{grad}

\DeclareMathOperator{\supp}{supp}

\begin{document}
\title[]{A topological classification of convex bodies}
\author[G. Domokos, Z. L\'angi \and T. Szab\'o]{G\'abor Domokos, Zsolt L\'angi \and T\'\i mea Szab\'o}

\address{G\'abor Domokos, Dept. of Mechanics, Materials and Structures, Budapest University of Technology,
M\H uegyetem rakpart 1-3., Budapest, Hungary, 1111}
\email{domokos@iit.bme.hu}
\address{Zsolt L\'angi, Dept. of Geometry, Budapest University of Technology,
Egry J\'ozsef u. 1., Budapest, Hungary, 1111}
\email{zlangi@math.bme.hu}
\address{T\'\i mea Szab\'o, Dept. of Mechanics, Materials and Structures, Budapest University of Technology,
M\H uegyetem rakpart 1-3., Budapest, Hungary, 1111}
\email{tszabo@szt.bme.hu}


\begin{abstract}
The shape of homogeneous, generic, smooth convex bodies as described by the Euclidean distance with nondegenerate critical points, measured from the center of mass represents a rather restricted class $\mathcal{M}_C$ of Morse-Smale functions on $\S^2$. Here we show that even $\mathcal{M}_C$ exhibits the complexity known for general Morse-Smale functions on $\S^2$ by exhausting all combinatorial possibilities: every 2-colored quadrangulation of the sphere is isomorphic to a suitably represented Morse-Smale complex associated with a function in $\mathcal{M}_C$ (and vice versa). We prove our claim by an inductive algorithm, starting from the path graph $P_2$ and generating convex bodies corresponding to quadrangulations with increasing number of vertices by performing each combinatorially possible vertex splitting by a convexity-preserving local manipulation of the surface. Since convex bodies carrying Morse-Smale complexes isomorphic to $P_2$ exist, this algorithm not only proves our claim but also generalizes the known classification scheme in \cite{Varkonyi}. Our expansion algorithm is essentially the dual procedure to the algorithm presented by Edelsbrunner et al. in \cite{Edelsbrunner}, producing a hierarchy of increasingly coarse Morse-Smale complexes. We point out applications to pebble shapes. \end{abstract}

\maketitle

\section{Introduction}

The study of static equilibria of convex bodies was initiated by Archimedes \cite{Archimedes1} and attracts even current interest (cf. \cite{Conway}, \cite{Dawson}, \cite{DawsonFinbow}, \cite{Heppes} or \cite{Zamfirescu}).
In mathematical terms, a convex body $K\in \mathcal{K}$ can be characterized by the scalar distance function $r_K : \S^2 \to (0, \infty)$, measured from the center of mass of $K$.
Static equilibrium points coincide with critical points of $r_K$, characterized by $\nabla r_K=0$. We call $K$ \emph{generic} if $r_K$ is a Morse-Smale function, i.e. it has only non-degenerate critical points
and the stable and unstable manifolds of any two critical points (under the flow induced by $\nabla r_K$) are transverse \cite{Zomorodian}. If we denote the numbers of stable, unstable and saddle-type critical points by $S,U,H$, respectively, then
$\{S,U\}$ is called the \emph{primary equilibrium class} of $K$ and the number of saddles can be obtained via the Poincar\'e-Hopf formula as
\ben \label{poincare}
H=S+U-2.
\een
In \cite{Varkonyi}, the above classification was introduced and it was shown that the primary classification system is complete in the sense that there are no empty primary classes.
Our current motivation is to go beyond this result and establish the completeness according to a more refined classification for generic convex bodies,
based on the topological arrangement of equilibria, i.e., to show that every combinatorially possible arrangement physically also exists.

If $K$ is generic then the \emph{Morse-Smale complex} of $r_K$ consists of the intersections of the stable and unstable manifolds of each of the critical points and it is a CW complex of dimension 2 on $\S^2$, i.e its 1-skeleton is a graph  embedded in $\S^2$.
 One possible representation of this graph is a special, 3-colored quadrangulation  $Q^{3\star}(K)$ on $\S^2$  and we refer to it as the
\emph{primal representation} of the Morse-Smale complex. The vertices of $Q^{3\star}(K)$ correspond to the critical points and they can be colored according to the stability-type by $S,U,H$, respectively. According to \cite{Edelsbrunner},
$Q^{3\star}(K)$ is a special quadrangulation where the colors of the vertices on every quadrangle run in cyclic order $S, H, U,H$ and the numbers of colored vertices satisfy Eq.(\ref{poincare}).

Thoroughout the paper we refer to embedded graphs (i.e. drawings of graphs) simply as `graphs',
however, it is important to note that two distinct (non-ho\-me\-o\-mor\-phic) embedded graphs may be represented by identical (isomorphic) abstract graphs. In the current paper
we are only interested in distinguishing convex bodies associated with non-homeomorphic embedded graphs, however, a less refined classification can be constructed by
distinguishing between convex bodies associated with non-isomorphic abstract graphs. Following this concept,
we denote the set of such 3-colored quadrangulations by $\mathcal{Q}^{3\star}$ and we call $Q^{3\star}(K)\in \mathcal{Q}^{3\star}$ the \emph{tertiaryy equilibrium class} of $K$. Figure \ref{fig:et2} illustrates two convex bodies, their associated gradient fields, their Morse-Smale complexes and their tertiary classes. As an intermediate classification system between primary and tertiary, secondary equilibrium classes may be defined by the
associated abstract graph, however, this is not the topic of the current paper.

\begin{figure}[ht]
\includegraphics[width=\textwidth]{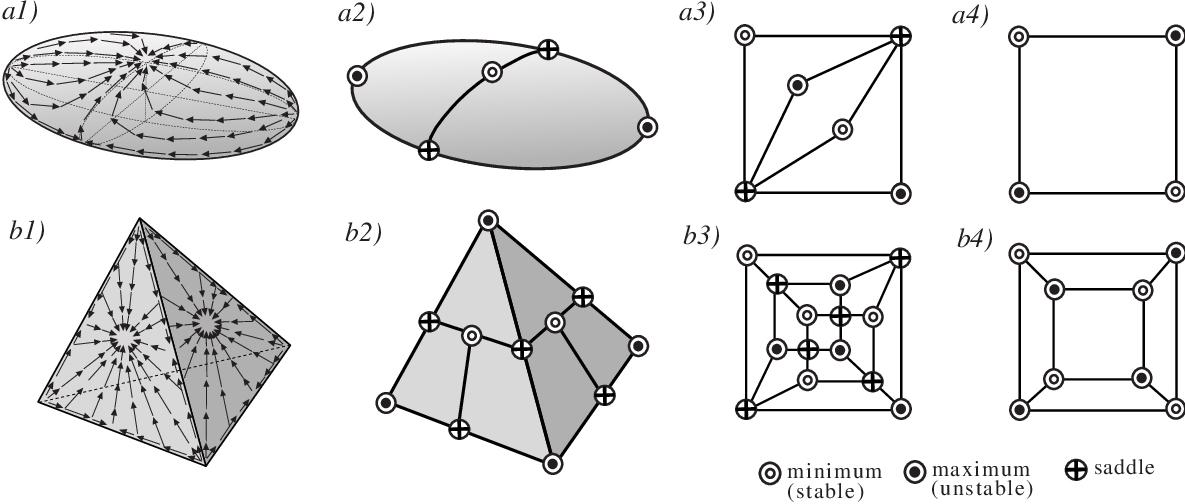}
\caption[]{Examples for equilibrium classes. a1) Gradient field of tri-axial ellipsoid in primary class $\{2,2\}$. a2) Morse-Smale complex associated with the tri-axial ellipsoid. a3) planar drawing of the graph representing the tertiary class of the tri-axial ellipsoid. a4) quasi-dual representation of the complex. b1) Gradient field of a smoothened tetrahedron $T$ in primary class $\{4,4\}$. b2) Morse-Smale complex associated with $T$. b3) planar drawing of the graph representing the tertiary class of $T$. b4) quasi-dual representation of the complex}
\label{fig:et2}
\end{figure} 

Our main goal is to prove that the tertiary classification is, similarly to the primary one, complete, i.e. that
\begin{thm}\label{thm:main}
For every $3$-colored quadrangulation $Q^{3\star} \in {\mathcal Q}^{3\star}$ of $\S^2$, there is a homogeneous convex body $K$, with a $C^\infty$-class boundary, such that $r_K$ is Morse-Smale, and the Morse-Smale complex of $r_K$ is 
isomorphic to $Q^{3\star}$ via an embedding preserving isomorphism.
\end{thm}
In order to explain the main idea of our proof we have to introduce an alternative, equivalent graph representation of the Morse-Smale complex.
By adding the $S-U$ diagonal edges to all $Q^{3\star} \in \mathcal Q^{3\star}$ and removing all $H$-colored vertices we introduce the \emph{quasi-dual} representation \cite{Dong} and we obtain the
set $\mathcal{Q}^2$ of two-colored quadrangulations on $\S^2$ (cf. Figure \ref{fig:quasi-dual}). We denote the two-colored quadrangulation associated with the convex body $K$ by $Q^2(K)$.  We also introduce ${\mathcal Q}^2_n\subset {\mathcal Q}^2$ for the 2-colored ($S,U$) quadrangulations of the sphere with exactly $S+U=n$ vertices and
the class of convex bodies with exactly $S+U=n$ extremal points will be denoted by $\mathcal{K}_n\subset \mathcal{K}$.
It is relatively easy to see (as we show in Section \ref{sec:graphs} based on results from \cite{Bagatelj}, \cite{Brinkmann} and \cite{Negami}) that by subsequent applications
of an operation called \em face contraction \rm (where two diagonal vertices in a quadrangular face are merged and the resulting double edges are also merged) an arbitrarily selected  two-colored quadrangulation $\bar Q^2 \in \mathcal{Q}^2_n$ can be
collapsed onto the path graph $P_2$ (the unique graph with 2 vertices and one edge \cite{Gross}) via a sequence of $(n-1)$ graphs. If we list the latter in reverse order, we obtain what we call \em combinatorial expansion sequences (for examples, cf. Figure~\ref{fig:metagraf} : \rm
\ben \label{gen}
Q^2_2\equiv P_2,Q^2_3,Q^2_4,\ldots,Q^2_{n-1},Q^2_{n}\equiv \bar Q^2 \mbox{ such that }
Q^2_i\in \mathcal{Q}^2_i.
\een
Subsequent elements of (\ref{gen}) are connected via \em vertex splittings \rm (cf.\cite{Bagatelj}, \cite{Brinkmann}), the inverse operation of face contraction. Both face contraction and vertex splitting
can be carried out on any $Q^2 \in \mathcal{Q}^2_n, n>2$, and thus the sequences (\ref{gen}) can be generated both forward (vertex splittings) and backward (face contractions). Neither the forward nor
the backward sequence is unique.
To prove Theorem \ref{thm:main}, in Section \ref{sec:geometry} we  show that each combinatorially possible vertex splitting can be realized on convex bodies by a convexity-preserving, local manipulation of the surface (consisting of two consecutive, local truncations) which we call \em equilibrium splitting \rm. By this operation we can create
\em geometric expansion sequences \rm
\ben \label{gen1}
K_2,K_3,K_4,\ldots,K_{n-1},K_{n}\equiv \bar K \mbox{ such that }
K_i\in \mathcal{K}_i,
\een
which are generated in such a way that they are linked to an arbitrarily, a priori given combinatorial expansion sequence (\ref{gen}) via
\ben
Q^2(K_i)=Q^2_i.
\een
We remark that a similar strategy of determining a sequence of geometric transformations appears also in the proof of Andreev's theorem \cite{Andreev}, \cite{RHD07}.
Combinatorial sequences belonging to an arbitrary graph $\bar Q^2 \in \mathcal{Q}^2_n$ can be generated by running (\ref{gen}) in reverse order, arriving at $P_2$.
Since convex bodies in class $\mathcal{K}_2$ carrying Morse-Smale complexes homeomorphic to $P_2$ exist (cf. \cite{Varkonyi}),
subsequently we can run (\ref{gen1}) forward to obtain the desired convex body $\bar K$ with $Q^2(\bar K)=\bar Q^2$.
This algorithm not only proves Theorem \ref{thm:main} but also generalizes
the primary scheme of \cite{Varkonyi}, which uses only the numbers $S,U$ of stable and unstable equilibria to classify convex shapes. 
We remark that an essential part of this generalization is the geometric truncation algorithm 
described in Section \ref{sec:geometry} since, as it was shown in \cite{Kapolnai}, combinatorial expansion sequences based on the the geometric truncations presented in \cite{Varkonyi} are \emph{not} capable of generating all tertiary equilibrium classes.
We also remark that, unlike combinatorial sequences (\ref{gen}), which can be run both forward and backward, our geometrical expansion sequences (\ref{gen1}) can be run only forward from an arbitrary initial convex body $K$.
Nevertheless, in a different setting, the geometric analogy of face contraction was considered in \cite{Bremer}.

The local truncations associated with each geometric expansion step are extremely delicate. To have the ability of performing \it any \rm combinatorially possible splitting
%in a geometrically consistent manner
(i.e. to obtain a convex body with the desired Morse-Smale complex) one has to render the vicinity of the given critical point \it arbitrarily sensitive \rm before the actual splitting is achieved with a planar truncation.
We achieve arbitrary sensitivity in Lemma ~\ref{lem:smallradius} by constructing a preliminary truncation with a sphere the radius of which is sufficiently close to the distance of the equilibrium point to be split from the center of mass.
If only the critical point is specified, however, the combinatorial structure of the splitting is arbitrary then the geometric task is substantially simpler \cite{Varkonyi}.

%\begin{figure}[ht]
%\includegraphics[width=\textwidth]{splitting5.eps}
%\caption[]{}
%\label{fig:split}
%\end{figure}

Our expansion algorithm is essentially the dual procedure to the one presented by Edelsbrunner et al. \cite{Edelsbrunner} (cf. also \cite{Dey}),
producing a hierarchy of increasingly coarse Morse-Smale complexes (i.e. running combinatorial expansion sequences in
reverse order compared to (\ref{gen})); a topic
attracting current interest in computational topology (cf. \cite{Bauer} or \cite{Gyulassy}).
Beyond illustrating that the tertiary classification scheme for generic convex bodies is complete
(i.e. the Morse-Smale complexes of generic convex bodies exhausts all possible combinatorial possibilities) and
offering a modest link between Morse theory and convex geometry, geometric expansion sequences defined in equation (\ref{gen1}) appear to be the
natural building blocks for the mathematical description of  pebble shape evolution under collisional abrasion.

Our paper is structured as follows. In Section \ref{sec:graphs} we introduce our combinatorial tools and show how Theorem~\ref{thm:main} follows from the main geometric lemma: Lemma~\ref{lem:main}. In Section \ref{sec:geometry}, we prove Lemma~\ref{lem:main} in three steps, formulated in Lemmas~\ref{lem:smoothing_subroutine}, \ref{lem:spherical_neighborhood} and \ref{lem:vertex_splitting}, and proved in Subsections~\ref{subsec:smooth}, \ref{subsec:sphere} and \ref{subsec:splitting}, respectively, by means of other auxiliary lemmas. We illustrate our results and discuss some related issues (including pebble abrasion) in Section \ref{sec:summary}.

\section{Preliminaries and the proof of Theorem~\ref{thm:main}}\label{sec:graphs}

In this section, first, we introduce the background for the proof, and, finally, show how Theorem~\ref{thm:main} follows from our main lemma: Lemma~\ref{lem:main}.
Throughout the paper, by the \emph{center} of a convex body $K$ we mean its center of mass.
Furthermore, the \emph{Morse-Smale complex of $K$} is meant to be the Morse-Smale complex defined on $\bd K$ by the
Euclidean distance function from the center of $K$. If we measure distance from a different point $w$,
then we write about the \emph{Morse-Smale complex of $K$ with respect to $w$}.

Let us recall that a \emph{quadrangulation} of $\S^2$ is the embedding of a finite graph on the 2-sphere such that it may have multiple edges and each face is bounded by a closed walk of length 4 (cf. \cite{Archdeacon} or \cite{Brinkmann}).
We note that this closed walk is permitted to contain the same vertex or edge more than once, and, following Archdeacon et al. \cite{Archdeacon}, we regard the path graphs (cf. \cite{Gross}) $P_2$ and $P_3$, the only trees with 2 or 3 vertices, respectively, as quadrangulations.
Let ${\mathcal Q}^0$ and ${\mathcal Q}$ denote the class of not-colored, and 2-colored quadrangulations, respectively. 
Observe that every quadrangulation is bipartite, and thus, 2-colorable, and that any quadrangulation (connected bipartite graph) can be colored in a unique way.
Furthermore, let ${\mathcal Q}^{3\star}$ be the class of 3-colored quadrangulations, with colors ($S,U,H$), satisfying $S+U-H=2$ (cf. \ref{poincare}),
and $deg(v)=4$ for any $v\in H$.

An important tool in describing Morse-Smale complexes of $\S^2$ is the so-called Quadrangle Lemma from \cite{Edelsbrunner}.

\begin{lem}[Edelsbrunner et al.]\label{lem:quadrangle}
Each region of the Morse complex is a quadrangle with maximum, saddle, minimum, and saddle point as vertices, in this order around the region. The boundary is possibly glued to itself along edges or vertices.
\end{lem}

Based on this lemma, a complete combinatorial description of a Morse-Smale complex can be given (\cite{Edelsbrunner}, \cite{Dong} and \cite{Zomorodian}):
it corresponds to a 3-colored quadrangulation in ${\mathcal Q}^{3\star}$, where colors correspond to the 3 types of non-degenerate critical points
(maxima, minima and saddles), satisfies the Quadrangle Lemma, and the degree of every saddle is $4$.
We follow Dong et al. \cite{Dong} and call this representation of the complex the \emph{primal} Morse-Smale graph.
Saddle points can be removed from the primal Morse-Smale graph without losing information:
first we connect maxima and minima in the quadrangles, then cancel saddle points and edges incident to them (see Figure~\ref{fig:quasi-dual}).
We call this representation the \emph{quasi-dual} Morse-Smale graph (cf. \cite{Dong}).

\begin{figure}[ht]
\includegraphics[width=1.1\textwidth]{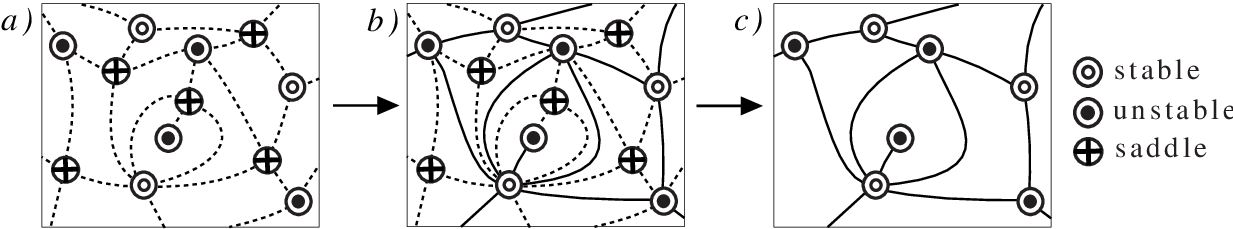}
\caption[]{Different representations of a Morse-Smale complex. a) Primal Morse-Smale graph in class ${\mathcal Q}^{3\star}$ b) Primal Morse-Smale graph with the maxima and minima connected c) Quasi-dual Morse-Smale graph in class ${\mathcal Q}$}
\label{fig:quasi-dual}
\end{figure}

The following lemma is a slight generalization of results of Bagatelj \cite{Bagatelj}, and also Negami and Nakamoto \cite{Negami}. To state it, for any guadrangulation $Q \in \mathcal{Q}^{0}$ and face $F$ of $Q$ with boundary walk $v_1e_1v_2e_2v_3e_3v_4e_4v_5$, where $v_1=v_5$ and edges are $e_i=\{v_i,v_{i+1}\}$, $i=1,2,3,4$ (cf. Figure~\ref{fig:face-contraction} a), left), we define the \emph{contraction of the face $F$} as the following:
shrink the region $F$ by identifying the vertices $v_1$ and $v_3$, the edges $e_1$ and $e_2$, and the edges $e_3$ and $e_4$.
Note that this operation on $Q$ is invertible; the inverse operation is called \emph{vertex splitting} (cf. Figure~\ref{fig:splittings}).

\begin{figure}[ht]
\includegraphics[width=0.6\textwidth]{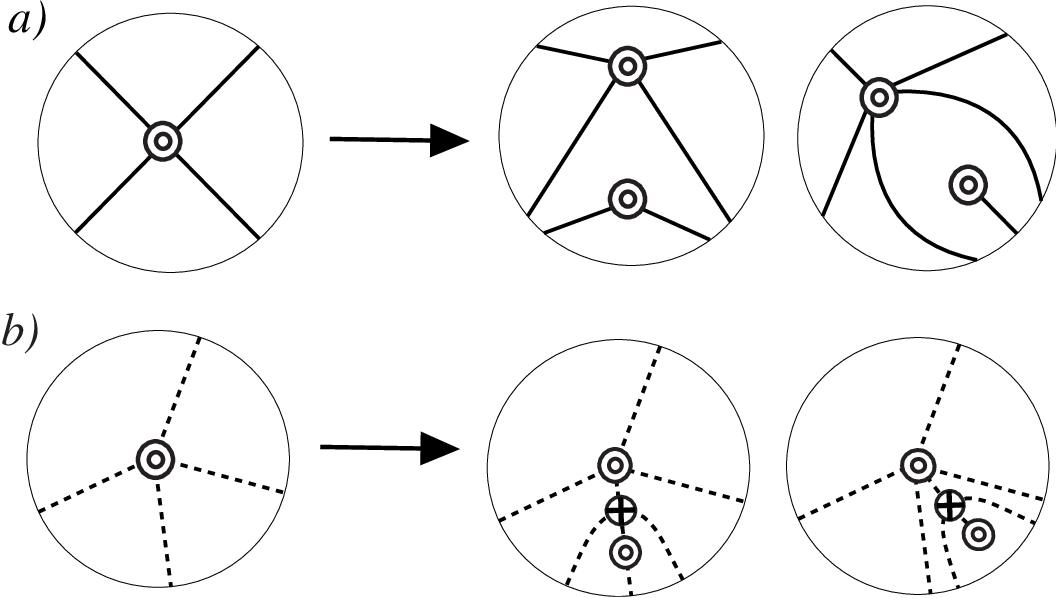}
\caption[]{a) Two splittings of a vertex in quasi-dual representation b) The same vertex splittings in primal representation}
\label{fig:splittings}
\end{figure}

\begin{lem}\label{lem:graph_expansion}
Any quadrangulation $Q \in \mathcal{Q}^{0}$ can be reduced to the path graph $P_2$ via a sequence of face contractions, or equivalently, $Q$ can be obtained from $P_2$ via a sequence of vertex splittings.
\end{lem}

\begin{proof}
Let $Q \in \mathcal{Q}^{0}$ and face $F$ of $Q$ with boundary walk $v_1e_1v_2e_2v_3e_3v_4e_4v_5$, where $v_1=v_5$ and edges are $e_i=\{v_i,v_{i+1}\}$, $i=1,2,3,4$.
If $Q$ is not simple, these vertices and edges may not be distinct, however,
the definition of quasi-dual representation (i.e. that these graphs have been generated by removing the saddle points of degree 4 from
the triangular representation) admits only two kinds of coincidences: two diagonal vertices $v_2$ and $v_4$ may coincide, and in this
case the edges $e_2$ and $e_3$ may coincide.
These two cases are shown in Figure~\ref{fig:face-contraction} b) and c), left.
Note that in Figure~\ref{fig:face-contraction} b) the internal domain bordered by the edges $e_2$ and $e_3$
is not a quadrangular face but necessarily contains additional vertices.

The proof is based on the observation, illustrated in Figure~\ref{fig:face-contraction}, that the contracted graph has one less vertex than $Q$, and is contained in the same class $\mathcal{Q}^{0}$. Since any graph with at least three vertices can be contracted, we may reduce $Q$ to the only graph $P_2 \in \mathcal{Q}^{0}$ with two vertices.
\end{proof}

\begin{figure}[ht]
\includegraphics[width=1.1\textwidth]{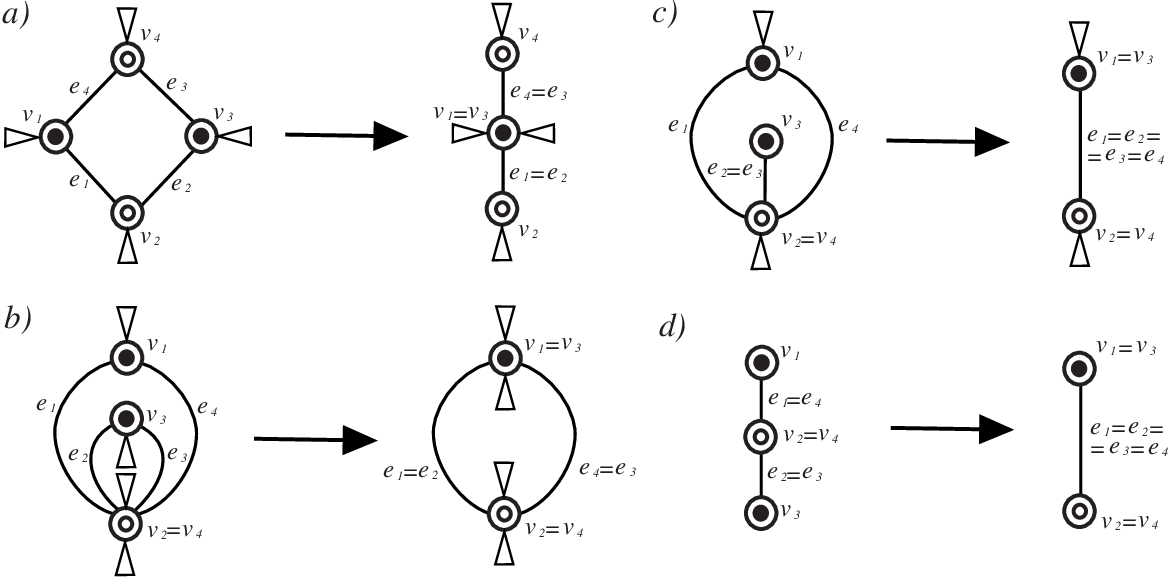}
\caption[]{Contraction of a face of a graph in class ${\mathcal Q}^2$. Like in \cite{Brinkmann}, triangles incident to some vertices indicate that one or more edges \emph{may} occur at that position around the vertex. Here, $v_1$ and $v_3$ are minimum points. The counterpart of the illustrated operation, which removes a maximum point from the graph, can be performed by switching the colors. a) Contraction of a generic face. b) Contraction of a face with a double vertex. c) Contraction of a face with a double edge. d) Contraction of the path graph $P_3$.}
\label{fig:face-contraction}
\end{figure}

\begin{rem}\label{rem:primalsequence}
As any quadrangulation $Q \in \mathcal{Q}^{0}$ can be colored with 2 colors in a unique way, Lemma~\ref{lem:graph_expansion} can be applied for 2-colored quadrangulations, i.e. for quasi-dual representations of Morse-Smale complexes. This process can be extended to primal graph representations as well via the natural identification of the elements of $\mathcal{Q}^{0}$ and ${\mathcal Q}^{3\star}$.
\end{rem}

Since a face contraction can be applied to any face of a quadrangulation, the sequence of graphs from $Q$ to $P_2$ is, in general, not unique.
In special applications additional criteria may be applied to single out one sequence among the combinatorially possible ones.
Face contraction on multigraphs was already used in \cite{Edelsbrunner} and \cite{Dey} to simplify Morse-Smale complexes. Edelsbrunner et al. used the primal representation of the complex in which a face contraction (defined in the quasi-dual representation) emerges as a double \emph{edge contraction} \cite{Diestel}. Their goal was to produce a hierarchy of increasingly coarse Morse-Smale complexes, therefore they applied an extra metric criteria (growing height differences on the edges of the graph) which resulted in an unambiguous sequence of graphs. This algorithm was also used in \cite{Domokos} to identify macroscopically perceptible static equilibrium points of 3D scanned pebbles.
\begin{defn}\label{defn:propstar}
If a convex body $K$, with $C^\infty$-class boundary, satisfies the property that
\begin{itemize}
\item the Euclidean distance function of $K$ is Morse-Smale;
\item $\bd K$ has nonzero principal curvatures at any critical point of this function;
\end{itemize}
then we say that $K$ satisfies property (*).
\end{defn}

Our main lemma is the following.

\begin{lem}\label{lem:main}
There is a convex body $K_0$ satisfying (*) with only one stable and only one unstable equilibrium point.
Furthermore, if $K$ is a convex body satisfying (*) with Morse-Smale graph $Q$, and $Q'$ is obtained from $Q$ via any vertex splitting, then there is a convex body $K'$ satisfying (*), with Morse-Smale graph $Q'$.
\end{lem}

We prove Lemma~\ref{lem:main} in Section~\ref{sec:geometry}.

\begin{proof}[Proof of Theorem~\ref{thm:main}]
Let $Q \in \mathcal{Q}^{3\star}$ be given. Then, by Lemma~\ref{lem:graph_expansion} and Remark~\ref{rem:primalsequence}, there is a sequence $Q_1,Q_2\ldots,Q_n \in \mathcal{Q}^{3\star}$ such that $Q_1 = P_2$, $Q_n = Q$ and for $i=1,2,\ldots,n-1$, $Q_{i+1}$ is obtained from $Q_i$ via a suitable vertex splitting.
Since the only graph in $\mathcal{Q}^{3\star}$ with two vertices is $P_2$, Lemma~\ref{lem:main} states that there is a convex body $K_1$, satisfying (*), with graph $P_2$. Thus applying Lemma~\ref{lem:main} $(n-1)$ times, we obtain that for any $i=2,3,\ldots,n$ (in particular, for $i=n$) there is a convex body $K_i$
satisfying (*) and having graph $Q_i$.
\end{proof}

\section{Proof of Lemma~\ref{lem:main}}\label{sec:geometry}

The proof is based on three lemmas. In their formulations and proofs we regard the Morse-Smale graph of a convex body as the drawing the primal representation of its Morse-Smale complex on the boundary of the body.

\begin{lem}[Smoothing subroutine]\label{lem:smoothing_subroutine}
If $K$ is a convex body with $C^1$-class boundary such that its distance function is Morse-Smale, every equilibrium point of $K$ has a $C^2$-class neighborhood,
and its principal curvatures are positive at any such point, then there is a $C^\infty$-class convex body $K'$ satisfying (*) such that $K$ and its Morse-Smale graph are arbitrarily small perturbations of $K'$ and its Morse-Smale-graph. Furthermore, there is a $C^\infty$-class convex body $K$ satisfying (*) with only one stable and one unstable point.
\end{lem}

\begin{lem}[Spherical neighborhood]\label{lem:spherical_neighborhood}
Let $K$ be a $C^\infty$-class convex body satisfying (*) and let $p$ be any stable or unstable point of $K$. Assume that the center of mass of $K$ is the origin $o$. If $p$ is stable, let $r > ||p||$, and if $p$ is unstable, $0 < r < ||p||$ be arbitrary. Then there is a convex body $K'$ satisfying (*) such that
\begin{itemize}
\item $K'$ and its Morse-Smale graph are arbitrarily small perturbations of $K$ and its Morse-Smale graph, respectively,
\item the stable/unstable point $p'$ of $K$ associated to $p$ has a spherical neighborhood in $\bd K'$ of radius $r'$ arbitrarily close to $r$.
\end{itemize}
\end{lem}

\begin{lem}[Vertex splitting]\label{lem:vertex_splitting}
Let $K$ be a $C^\infty$-class convex body satisfying (*), with Morse-Smale graph $Q \in \mathcal{Q}^{3\star}$. Let $p$ be a stable or unstable point of $K$, and $Q'$ be a graph obtained from $Q$ be splitting the vertex $p$ in any given way. Assume that $p$ has a spherical neighborhood in $\bd K$ with radius $r'$ sufficiently close to some suitable chosen $r$ with $r > ||p||$ if $p$ is stable and $0 < r < ||p||$ if $p$ is unstable.
Then there is a $C^\infty$-class convex body $K'$ with Morse-Smale graph $Q'$ satisfying (*).
\end{lem}

\subsection{Proof of Lemma~\ref{lem:smoothing_subroutine}}\label{subsec:smooth}

Let $o$ denote the center of mass of $K$, and observe that $o \in \inter K$.
We start with a lemma that we use a couple of times in the paper. To state it, for any convex body $K'$ containing $o$ in its interior, we say that $q \in \bd K$ and $q' \in \bd K'$ are \emph{corresponding points}, if $q'= \lambda q$ for some $\lambda > 0$. Furthermore, we let
$F_{K',x}: \bd K' \to \Re$ be the function defined by $q' \mapsto |q'-x|$, and if $x=o$ we write simply $F_{K'}$.

\begin{lem}\label{lem:approx}
Let the equilibrium points of $K$ be $p_1, p_2, \ldots, p_k$, and for $i=1,2,\ldots, k$, let $U_i$ be a closed neighborhood of $p_i$ in $\bd K$
such that $\bd K$ is $C^2$-class in $U_i$ and the eigenvalues of the Hessian $H_K$ of $F_K$ do not change signs in $U_i$.
Then there is an $\varepsilon > 0$ such that for any $x \in \Re^3$ with $|x| < \varepsilon$, the Morse-Smale complex of $K$
with respect to $o$ is homeomorphic to that of any $K'$ with respect to $x$ satisfying the property that:
\begin{enumerate}
\item any pair of corresponding points $q \in \bd K$ and $q' \in \bd K'$ satisfies $|q'-q| < \varepsilon$ and $|\grad F_{K'} (q') - \grad F_K (q)| < \varepsilon$,
\item for any value of $i$, $\bd K'$ is $C^2$-class at any point corresponding to a point $q \in U_i$,
\item for any value of $i$ and any pair of corresponding points $q \in U_i$ and $q' \in \bd K'$, the two eigenvalues of the Hessian $H_{K'}$ at $q'$ differ less than $\varepsilon$ from the two eigenvalues of $H_K$ at $q$, respectively.
\end{enumerate}
\end{lem}

\begin{proof}
For any value of $i$, let $S_i$ be a simple, closed continuous curve in $U_i$ around $p$, transversal to any integral curve of the gradient flow of $K$ intersecting it, and separating $p_i$ in $\bd K$ from any other equilibrium point of $K$.
For simplicity, we assume that $Si = \bd U_i$.
Let $\delta > 0$ satisfy the property that $| \grad F_K (q) | \geq \delta$ for any $q \in (\bd K) \setminus \bigcup_{i=1}^k U_i$, and
the absolute values of the eigenvalues of $H_K$ at any $q \in U_i$, for any value of $i$, are at least $\delta$.

For any convex body $K'$ containing $o$ in its interior, let $S_i'$ and $U_i'$ be the central projections of $S_i$ and $U_i$ from $o$ onto $\bd K'$, respectively. Then there is some $\varepsilon_1 > 0$ such that if $K'$ satisfies the conditions of the lemma with $\varepsilon_1$ in place of $\varepsilon$, $| \grad F_{K'} (q') | \geq \frac{\delta}{2}$ for any $q' \in (\bd K') \setminus \bigcup_{i=1}^k U_i'$, and
the absolute values of the eigenvalues of $H_{K'}$ at any $q' \in U_i'$, for any value of $i$, are at least $\frac{\delta}{2}$.
Since the function $|r-x|$ is a smooth function of $x$, it follows that there is some $0 < \varepsilon_2 \leq \varepsilon_2$ such that
if $K'$ satisfies the conditions of the lemma with $\varepsilon_2$ in place of $\varepsilon$ and $|x| < \varepsilon_2$, then $| \grad F_{K',x} (q') | \geq \frac{\delta}{4}$ for any $q' \in (\bd K') \setminus \bigcup_{i=1}^k U_i'$, and the absolute values of the eigenvalues of the Hessian $H_{K',x}$ at any $q' \in U_i'$, for any value of $i$, are at least $\frac{\delta}{4}$.
Furthermore, we may assume that the signs of the eigenvalues of $H_{K',x}$ in $U_i'$ are equal to the signs of the eigenvalues of $H_K$ in $U_i$, and that at every point of $S_i'$, $\grad F_{K',x}$ points inside if $p_i$ is a stable point, and outside if $p_i$ is unstable.

Thus, each equilibrium point of $K'$, with respect to $x$, is contained in $U_i'$ for some value of $i$, and the type of any equilibrium point in $U_i'$ is the same as that of $p_i$.
Furthermore, as the winding number of $S_i'$ is $1$, if $p_i$ is stable or unstable, then $U_i'$ contains exactly one equilibrium point.

Finally, for any pair $p_i$, $p_j$ connected by an edge in the quasi-dual representation of the Morse-Smale complex of $K$, where $p_i$ is stable and $p_j$ is unstable, let $q_{i,j}$ be a point in $S_i$ such that the integral curve through $q_i$ ends at $q_j$.
Let $0 < \varepsilon \leq \varepsilon_2$ be chosen such that for any such $q_{i,j}$, the integral curve of the gradient flow of any $K'$, with respect to any $x$ with $|x| < \varepsilon$, through the point $q_{i,j}'$ corresponding to $q_{i,j}$, meets $S_j'$. From this it follows that every edge in the quasi-dual representation of the Morse-Smale complex of $K$ corresponds to an edge of the quasi-dual Morse-Smale graph of $K'$ with respect to $x$. Since any quasi-dual graph is a quadrangulation, the graph of $K'$ do not have additional edges. Thus, the graphs of $K$ and $K'$ are homeomorphic.
\end{proof}

Let $f:\Re^3 \to \Re$ denote the \emph{distance function} of $K$; that is, $f$ is defined by
$f(u) = \min \{ \lambda : \lambda \geq 0, u \in \lambda K \}$ (cf. \cite{BF87}).
Let $h:\Re^3 \to \Re$ be a nonnegative, $C^\infty$-class function such that $\supp h \subseteq \B$
and $\int_{\Re^3} h(x) \dif x = 1$.
Such a function is called by various names in the literature: \emph{mollifier} (cf. \cite{E98}),
or \emph{bump function} (cf. \cite{H76}) or \emph{probability distribution}.
Clearly, we may choose $h$ in a way that its symmetry group is $O^3$.
Observe that by setting $h_t (x) = \frac{h(x/t)}{t^3}$, we obtain a family of
$C^\infty$-class functions with $\supp h_t \subseteq t\B$
and $\int_{\Re^3} h_t(x) \dif x = 1$.

Let us define the function $g_t : \Re^3 \to \Re$ as the convolution
\[
g_t(x) = \int_{\Re^3} f(x-y) h_t(y) \dif y .
\]

Clearly, for every $t$, $g_t$ is $C^\infty$-class, and as the integral average of convex functions is
convex, it is convex.
In particular, it follows that the set $K_t = g_t^{-1}([0,1])$ is compact and convex, and hence
it is a convex body for sufficiently small values of $t$.
Furthermore, by \cite[Theorem 2.3]{H76}, $D^k_x(g_t)(z_1,z_2,\ldots,z_k)=\int_{\Re^3} f(y) D^k h_t(x-y) (z_1,z_2,\ldots,z_k) \dif y$
for any $z_1,z_2,\ldots, z_k \in \Re^3$ and $x$ being contained in a neighborhood of $\bd K_t$, where $D^k$ denotes the $k$-linear $k$th derivative functional.
Thus, we may apply the Implicit Function Theorem for $k$ times continuously differentiable functions, from which it follows that $\bd K_t$ is a $C^\infty$-class submanifold of $\Re^n$.
In addition, it also follows (cf. \cite{G02} or \cite{H76}), that if
$U$ is a compact set and $f$ is $C^n$-class on $U$ for some $2 \leq n < \infty$, then on $U$
$g_t$ converges uniformly to $f$, together with its derivatives up to order $n$,
as $t \to 0$.
Thus, by Lemma~\ref{lem:approx},  for sufficiently small values of $t$, that is, for any $t \in [0,\varepsilon]$ for some $\varepsilon > 0$, the Morse-Smale complex of $K_t$, with respect to any point in a fixed neighborhood of $o$ of radius $\rho > 0$, is homeomorphic to that of $K$ with respect to $o$; and $\bd K_t$ has nonzero principal curvatures at the critical points of $K$.
This implies that $K_t$ satisfies (*).
Note that the center $o_t$ of $K_t$ tends to $o$ as $t \to 0$.
Thus, if $t$ is sufficiently small, $| o_t | < \rho$.
Choosing $t$ with this property and setting $K'=K_t$, the assertion readily follows.

To prove the second part, we apply our method for the mono-monostatic convex body $K_0$ constructed in \cite{Varkonyi}.
This body is $C^1$-class at every boundary point, and, apart from the two equilibrium points, it is $C^2$-class.
Furthermore, at the two equilibrium points the one-sided curvatures in every normal section are positive. Thus, applying our method yields
the assertion.

\begin{rem}\label{rem:smoothing_sphere}
Note that as the symmetry group of $h$ is $O^3$, if $p \in \bd K$ has a spherical cap neighborhood, then so does the corresponding point of $K'$.
\end{rem}

\subsection{Proof of Lemma~\ref{lem:spherical_neighborhood}}\label{subsec:sphere}

Let $o$ be the center of mass of $K$, and assume that $p=(0,0,1)$.
Then the condition that $p$ is nondegenerate is equivalent to saying that
$\kappa_1 \neq 1 \neq \kappa_2$, where $\kappa_1$ and $\kappa_2$ are the two principal
curvatures of $\bd K$ at $p$.

Now we define a two-parameter family of truncations of $K$, denoted by $K_R(\varepsilon)$.
Let $z=f(x,y)$ be a function the graph of which is a neighborhood of $p$ in $\bd K$.
Furthermore, for any $\phi \in [0,\pi]$, let $f_L(t)=f(t\cos \phi,t\sin \phi)$ be the restriction
of $f$ to the line $L$ containing $(\cos \phi, \sin \phi)$ and the origin.
Consider a value of $R$ such that $\frac{1}{R}$ is strictly smaller than
any of the two principal curvatures of $\bd K$ at $p$.
Then $\frac{1}{R} < |f''_L(0)|$ for any line $L$.
Since $\bd K$ is $C^\infty$-class in a neighborhood of $p$, so is $f$ in a neighborhood of $(0,0)$.
Thus, for some $\delta > 0$, $\frac{1}{R} < f''_L(t)$ for any $|t| \leq \delta$.

We choose $\delta$ satisfying this condition, and consider the part $S$, with points for the coordinates of which $\sqrt{x^2+y^2} \leq \delta$ and $z > 0$, of the sphere of radius $R$ that touches the plane $\{ z=1 \}$
at $p$ from the side containing the origin. Clearly, $S$ is a closed spherical cap.
Now, set $C_R$ as the intersection of all closed half spaces that contain $S$ and with their boundaries tangent to $S$.
Then $C_R$ is a `cone with a rounded apex'.
Note that $C_R$ is convex, contains $K$, and that $(\bd C_R) \cap K = \{ p \}$.
Furthermore, the boundary of the translate $C_R(\varepsilon)$,
with a sufficiently small vector $(0,0,-\varepsilon)$ where $\varepsilon > 0$,
intersects $\bd K$ in a simple smooth closed curve contained in the translate of the
interior of the spherical cap $S$; here, the fact that the intersection is a simple closed curve is shown by the fact that the planar section of the boundary of the translate with any plane through $[o,p]$ intersects the corresponding planar section of $\bd K$ at two distinct points (cf. Figure~\ref{fig:cutting}).
Our first lemma describes the properties of the intersection $K_R(\varepsilon) = C_R(\varepsilon) \cap K$.

\begin{figure}[ht]
\includegraphics[width=0.45\textwidth]{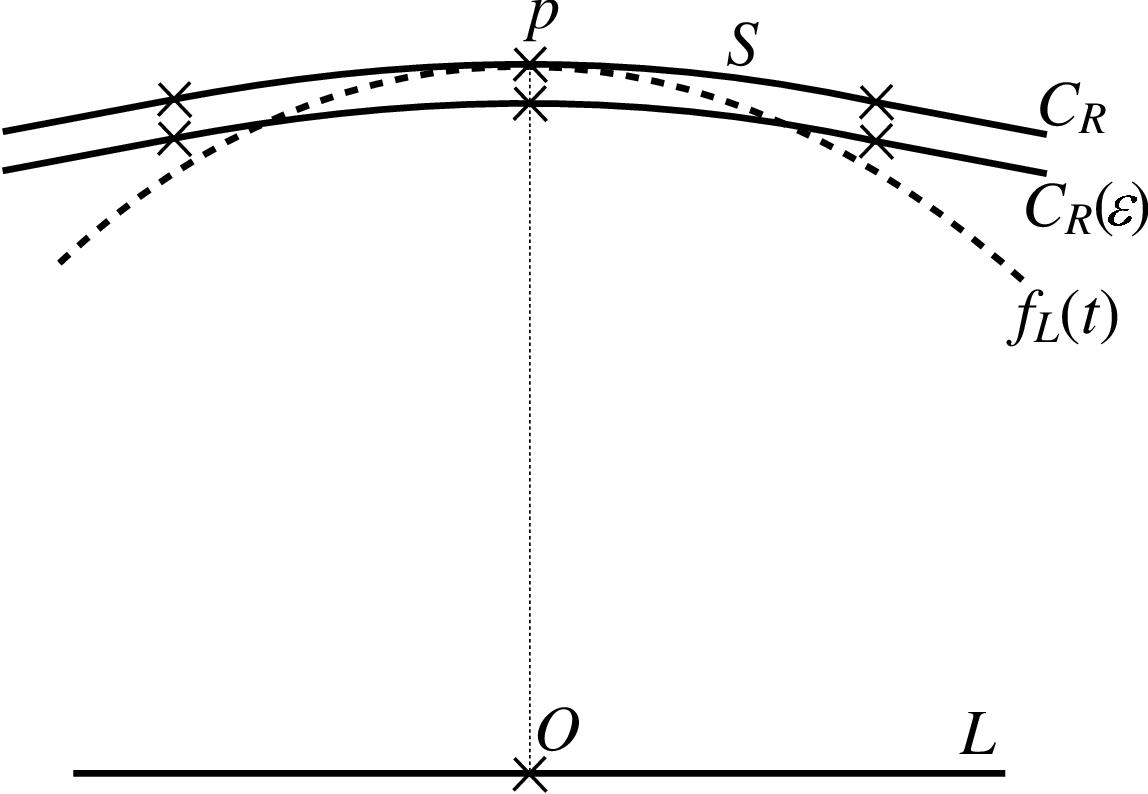}
\caption[]{Truncating $K$ with $C_R(\varepsilon)$}
\label{fig:cutting}
\end{figure}

\begin{lem}\label{lem:truncationbysphere}
Let $\frac{1}{R}$ be smaller than any of the two principal curvatures of $\bd K$ at $p$.
Then there are constants $\lambda_1, \lambda_2, \lambda_3 > 0$ depending only on $R$ and $K$ such that
for every sufficiently small $\varepsilon > 0$ the following hold.
\begin{itemize}
\item[(1)] no point of $(\bd K) \cap \B_{\lambda_1 \sqrt{\varepsilon}}(p)$ belongs to $K_R(\varepsilon)$.
\item[(2)] $(\bd K) \setminus \B_{\lambda_2 \sqrt{\varepsilon}}(p) \subseteq K_R(\varepsilon)$.
\item[(3)] If $D$ is a convex body, with center $o'$ and satisfying $K_R(\varepsilon) \subseteq D \subseteq K$,
then $|o'| \leq \lambda_3 \varepsilon^2$.
\end{itemize}
\end{lem}

\begin{proof}
First, observe that $K_R(\varepsilon)$ is convex, as it is the intersection of convex bodies.

Let $f$ be the two-variable function whose graph describes $\bd K$ near $p=(0,0,1)$.
Like there, let $f|_L : \Re \to \Re$ denote its restriction $f|_L (t) = f(t \cos \phi, t \sin \phi)$,
where $ \phi \in [ 0, \pi ]$, on a line $L$ passing through the origin and $(\cos \phi, \sin \phi)$.
Let $h$ denote the one-variable function, defined on $[-R,R]$, the graph of which is a semicircle, of radius $R$, and with maximum $h(0)=1-\varepsilon$.
By the conditions in Lemma~\ref{lem:truncationbysphere}, we have $0 > h''(0) > f|_L''(0)$.
Hence, from the second-degree Taylor polynomials of $f|_L$ and $h$, we obtain that
\[
\varepsilon - \lambda_1^\star t^2 < f_L(t)-h(t) < \varepsilon - \lambda_2^\star t^2
\]
for some positive constants $\lambda_1^\star, \lambda_2^\star$.
Thus, $(1)$ and $(2)$ clearly follow with $\lambda_1 = \frac{1}{\lambda_1^\star}$ and
$\lambda_2 = \frac{1}{\lambda_2^\star}$.

Now we show (3). Let $o'=(x',y',z')$.
Recall that
\[
o' = \frac{\int_{q \in D} q \dif V}{\int_{q \in D} 1 \dif V} .
\]

First, we estimate $x'$.
Note that the part of $K$ outside $D$ can be covered by an axis-parallel brick of side-lengths
$2\lambda_2 \sqrt{\varepsilon}, 2\lambda_2\sqrt{\varepsilon}$ and $c \varepsilon>0$
for some constant $c$ independent of $\varepsilon$.
Thus,
\[
\left| \int_{q \in D} 1 \dif V - \int_{q \in K} 1 \dif V \right| \leq 4\lambda_2 ^2 c \varepsilon^2,
\]
and
\[
\left| \int_{q \in D} x \dif V - \int_{q \in K} x \dif V \right| \leq 4\lambda_2^3 c \varepsilon^{5/2},
\]
from which $|x'| \leq \lambda_3'\varepsilon^2$ follows.
In the same way, we may obtain similar bounds for $y'$ and $z'$, which readily yields the assertion.
\end{proof}

We note that Lemma~\ref{lem:truncationbysphere} can be applied for the degenerate case $R=\infty$ as well.
The next lemma guaranties that the truncated body $K_R(\varepsilon)$ has the same numbers and types of static equilibrium points.
The proof is based on the idea presented in the proof of Lemma~\ref{lem:approx}.

\begin{lem}\label{lem:truncation_contEQpoints}
Let $p$ be a stable point of $K$, or let $p$ be an unstable point and $R < ||p||$.
If $R$ satisfies the conditions in Lemma~\ref{lem:truncationbysphere},
then there is an $\varepsilon > 0$ such that for every $t \in [0,\varepsilon]$ the following holds:
$K_R(t)$ and $K=K_R(0)$ have the same number of stable/unstable and saddle points. Furthermore,
the coordinates of these points are continuous functions of $t$.
\end{lem}

\begin{proof}
Note that as the Euclidean distance function from a fixed point is $C^\infty$-class at every point but the origin,
all the partials of any order at any point of $\bd K$ change continuously when translating $K$.
Hence, applying the Poincar\'e-Hopf Theorem \cite{Arnold} to a compact neighborhood of any equilibrium point
shows that the numbers of the stable/unstable/saddle points of $K$ do not change under a translation by a small vector,
and their coordinates are continuous functions of the translation vector.
Similarly, the gradient vector field changes continuously under a translation of $K$.
Thus, $o$ has a neighborhood $U$ such that for any $u \in U$, the Morse-Smale complex
of $K$ is homeomorphic to that of $K+u$, or, in other words, the Morse-Smale complex of $K$
with respect to $u$ is homeomorphic to that with respect to $o$.

If $t$ is sufficiently small, then the center $o_t$ of $K_R(t)$ is contained in $U$.
Hence, all the stable/unstable and saddle points of $K$ but $p$ change continuously as functions of $t$.
Note that by (1) and (3) of Lemma~\ref{lem:truncationbysphere},
the stable/unstable point of $K$, with respect to $o_t$, that corresponds to $p$
is contained in the part truncated by $C_R(t)$.
Thus, this point does not belong to $K_R(t)$.
On the other hand, by (2) of the same lemma, there is exactly one equilibrium point of $K_R(t)$
on the part belonging to $C_R(t)$.

Finally, it can be seen geometrically that no plane, perpendicular to $[u,q]$ for some $u \in V$,
supports $K_R(t)$ at a point $q \in (\bd K) \cap  \bd C_R(t)$, and thus, there are no more equilibrium points of $K_R(t)$.
\end{proof}

Now we show that there is a value of $R$ for which the assertion in Lemma~\ref{lem:spherical_neighborhood} holds.
To do this, consider some $K_R(\varepsilon)$, where $\varepsilon > 0$ and $R>0$
satisfy the conditions of Lemma~\ref{lem:truncation_contEQpoints}.
By this lemma, we may assume that $K_R(\varepsilon)$ has the same numbers of stable/unstable and saddle
points as $K$ does.
Let $o_R$ denote the center of $K_R(\varepsilon)$, and $p_R$ be the critical point of $K_R(\varepsilon)$
corresponding to $p$.
Consider a simple closed $C^1$-class curve $g \subset \bd K_R(\varepsilon)$, separating $p_R$ from
all the other critical points of $K_R(\varepsilon)$, which is disjoint from the truncated,
spherical part, and is \emph{not} tangent to any integral
curve of $\bd K$, with respect to $o_R$, that intersects $g$.
Without loss of generality, we may assume that the central projection of $g$ onto the unit sphere $\S^2$ from $o$ is also a simple, closed, $C^1$-class curve.

Now we show that even though $\bd K_R(\varepsilon)$ has nonsmooth boundary and thus cannot have an associated Morse-Smale complex,
if we smooth $K_R(\varepsilon)$ using the function $h_t$ from the proof of Lemma~\ref{lem:smoothing_subroutine}, then
the Morse-Smale complex of the resulting body $K'$ is homeomorphic to that of $K$.
Let $g'$ be the central projection of $g$ from $o$ to $\bd K'$.
Observe that as the gradient vector field changes continuously as a function of $t$,
for sufficiently small $t$ $g'$ is not tangent to any integral curve of $K'$ intersecting it.
Furthermore, any such integral curve ends at a critical point in the region inside $g'$
and for small values of $\varepsilon$ and $t$ there is a unique critical point in this region, which we denote by $p'$.
Thus, we may choose values of $\varepsilon$ and $t$ such that the Morse-Smale complex of $K'$
is homeomorphic to that of $K$.
By Remark~\ref{rem:smoothing_sphere}, $p'$ has a spherical neighborhood, and the radius of this sphere
is arbitrarily close to $R$.

As a result of our consideration, we may assume that for \emph{some} value of $R$, the examined critical point $p$ has
a spherical cap neighborhood in $\bd K$ of radius $R$.
In the last lemma of the subsection, we show that this truncation can be carried out using a sphere of essentially any radius.
To formulate it, recall that the integral curves connecting $p$ to a saddle point of $K$ are denoted by $\Gamma_1, \Gamma_2, \ldots, \Gamma_k = \Gamma_0$ in cyclic order around $p$.

\begin{lem}\label{lem:smallradius}
Let $K$ satisfy (*), and assume that $p$ has a spherical cap neighborhood in $\bd K$, of radius $R$.
Let $0 < r < ||p||$ if $p$ is an unstable point, and $r > ||p||$ if $p$ is a stable point,
and let $\delta > 0$ arbitrary.
Then there is a convex body $K' \subseteq K$ satisfying (*) and the following.
\begin{itemize}
\item The Morse-Smale complex of $K'$ is homeomorphic to that of $K$.
\item Denoting the critical point of $K'$ corresponding to $p$ by $p'$, $p'$ has a spherical cap neighborhood
in $\bd K'$, of radius arbitrarily close to $r$.
\item Denoting the integral curve of $K'$ corresponding to $\Gamma_i$ by $\Gamma'_i$ for every $i$, and
by $t'_i$ the unit tangent vector of $\Gamma'_i$ at $p'$, we have that $||t'_i - t_i || < \delta$.
\end{itemize}
\end{lem}

\begin{proof}
First, let $r < R$. Then we set $p'=p$.
Let $U$ be a (sufficiently small) neighborhood of $o$ such that
for every $u \in U$, the Morse-Smale complex of $K$ with respect to $u$ is homeomorphic
to that with respect to $o$. By (3) of Lemma~\ref{lem:truncationbysphere},
there is an $\varepsilon > 0$ such that for any convex body $K'$, satisfying
$(K \setminus \B_\varepsilon(p)) \subseteq K' \subseteq K$,
the center of $K'$ is contained in $U$.

Clearly, we may replace the part of $\bd K$ in $\B_\varepsilon(p)$ with a surface of rotation $S'$, with the line containing $[o,p]$ as its axis of symmetry, in such a way that:
\begin{itemize}
\item $p \in S' \subset K$ and $(\bd K \setminus \B_\varepsilon(p)) \cup S'$ is the boundary of a convex body $K'$,
\item the modified body $K'$ has a $C^1$-class boundary,
\item a neighborhood of $p$ in $S'$ belongs to a sphere of radius $r$,
\item there is no equilibrium point on $S'$ but $p$, with respect to any point of $U$ on the line containing $[o,p]$.
\end{itemize}

Since $p$ has a spherical neighborhood in $\bd K$, the center of $K'$ lies on the line containing $[o,p]$, and thus all the integral curves emanating from
$p$ in $\B_\varepsilon(p)$ are the meridians of $S'$.
Let these curves be denoted by $\Gamma'_1, \ldots, \Gamma'_k = \Gamma'_0$ in cyclic order, with $t'_i$ denoting
the unit tangent vector of $\Gamma'_i$ at $p$.
Thus, since the integral curves of $K$ with respect to $u \in U$, connecting $p$ and the saddle points of $K$
are deformed continuously when modifying $u$, the assertion readily follows for sufficiently small $U$. To obtain a convex body with a $C^\infty$-class boundary, we may finally apply Lemma~\ref{lem:smoothing_subroutine}.

If $r > R$, we may use a truncation by a sphere of radius $r$, like in the previous part of the subsection.
\end{proof}

%\begin{rem}\label{rem:gomboc}
%Observe that the two critical points of the convex body $K$ with Morse-Smale complex $P_2$
%are the `poles' of the parametrization (cf. \cite{Varkonyi}), and thus, may not have $C^\infty$-class neighborhoods.
%Nevertheless, the curves $P \cap \bd K$, where $P$ is any closed half plane containing a critical point of $K$ and its center in its relative boundary,
%are $C^\infty$-class curves with nonzero second derivatives at the critical points. Thus, a straightforward modification of our arguments can be applied in this case.
%\end{rem}

\subsection{Proof of Lemma~\ref{lem:vertex_splitting}}\label{subsec:splitting}

Examples of geometric truncations generating some of the combinatorially possible vertex splittings are shown in Figure~\ref{fig:splitting}.
Note that in a vertex splitting, the edges around the vertex to be split are partitioned into two parts (one of which is possibly empty), in cyclic order, 
such that edges in the same part end up at the same split vertex (cf. Figure~\ref{fig:splitting}, panel (a)).

\begin{figure}[ht]
\includegraphics[width=0.7\textwidth]{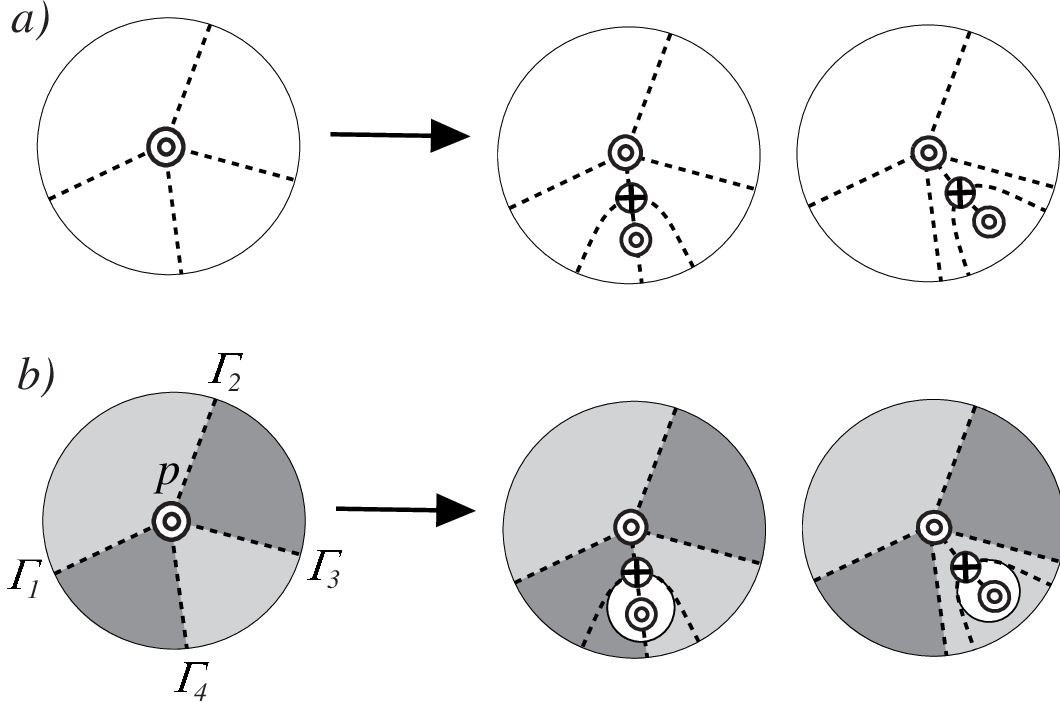}
\caption[]{Vertex splittings at vertex $p$. Here, $p$ is a minimum point (of course, $p$ could also be a maximum).
a) Vertex splitting in the primal representation in class ${\mathcal Q^{3\star}}$ ($d(v)=4$)
b) Geometric realization with a truncation by a plane}
\label{fig:splitting}
\end{figure}

First, we consider the case that  vertex $p$ to be split is a stable point of $K$. The truncation is illustrated in Figure~\ref{fig:geometry}.

\begin{figure}[ht]
\includegraphics[width=1.1\textwidth]{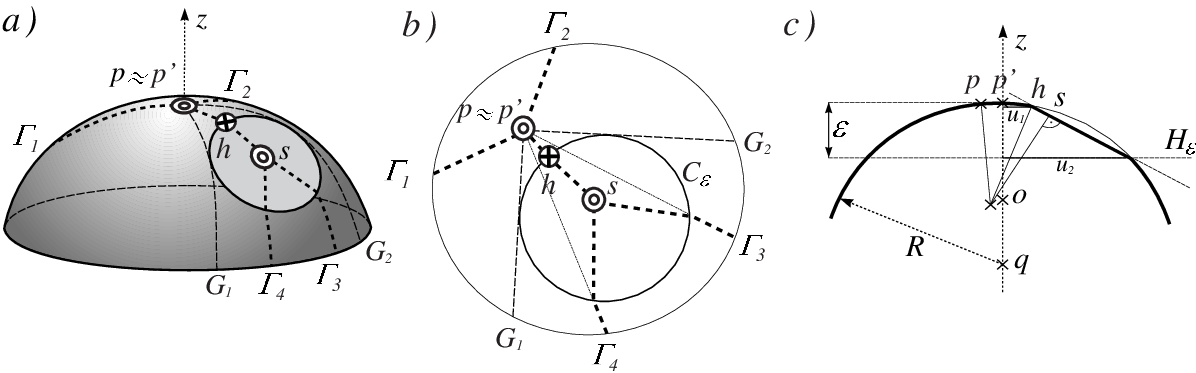}
\caption[]{Truncation with a plane}
\label{fig:geometry}
\end{figure}

Let the center of mass of $K$ be the origin, and set $p=(0,0,1)$.  According to the formulation of Lemma~\ref{lem:spherical_neighborhood}, we may choose the radius $r'$ of the spherical neighborhood of $p$ arbitrarily close to any given value $r > ||p||=1$. We choose the value of $r$ later.

Let the integral curves of $K$, connecting $p$ and the saddle points of $K$, in cyclic order around $p$, be $\Gamma_i$, where $i=1,2,\ldots,k$, and the unit tangent vector of $\Gamma_i$ at $p$ is $t_i$.
Let $R$ and $q$ denote the radius and the center of the sphere, containing a spherical cap neighborhood of $p$.
We denote this neighborhood by $S$.
Then, according to Lemma~\ref{lem:spherical_neighborhood}, we may assume that $R$ is arbitrarily close to any given number greater than $||p||$.
Thus, the arc of any integral curve of $K$ within $S$, including the $\Gamma_i$s, are great circle arcs.

Let $G_1$ and $G_2$ be two open great circle arcs in $S$, starting at $p$, that are in the interiors of the two regions (in the degenerate case, on region), bounded by the $\Gamma_i$s, that we want to split.
These different regions are shaded in Figure~\ref{fig:splitting} b).
We may choose $G_1$ and $G_2$ in a way that they do not belong to the same great circle.
We intend to truncate $K$ near $p$ with a plane in such a way that there is a new stable point on the plane, a new saddle point on its boundary, and the saddle point is connected to the stable point corresponding to $p$, to the stable
point on the plane and to the two unstable points in the boundaries of the regions containing $G_1$ and $G_2$.
To do this, we use a planar section that touches $G_1$ and $G_2$.
To examine the properties of these sections, we first prove a technical lemma that we are going to use in the construction.

\begin{lem}\label{lem:limit}
Let $C$ be the unit circle in the plane $\Re^2$ centered at the origin. Let $p=(0,t)$ with $t>0$.
Let $q_1=(u_1,v_1)$ and $q_2=(u_2,v_2)$ be two points of $C$ such that $v_1>0$.
\begin{itemize}
\item[(1)] If $q_2$ is defined by the property, as a function of $q_1$, that $[q_1,q_2]$ is perpendicular to $[p,q_1]$,
then $\lim\limits_{u_1 \to 0} \frac{u_2}{u_1} = \frac{1+t}{1-t}$.
\item[(2)] If $q_2$ is defined by the property, as a function of $q_1$, that the angle of $[q_1,q_2]$ and $[p,q_1]$ is $\frac{\pi}{2}-c u_1$ for some fixed constant $c$ independent of $u_1$, then $\lim\limits_{u_1 \to 0} \frac{u_2}{u_1} = \frac{1+t}{1-t} + 2c$.
\end{itemize}
\end{lem}

\begin{proof}
We note that $u_2 = \frac{(1-t^2)u_1}{1-2\sqrt{1-u_1^2}t+t^2}$, from which (1) immediately follows. The second part
can be proven with a similar elementary computation.
Note that in the second case the angle between $[p,q_1]$ and the segment connecting $p$ and its orthogonal projection
on $[q_1,q_2]$ is $c u_1$.
\end{proof}

Now, consider a spherical circle on $S$ that touches $G_1$ and $G_2$, and let $u_1$ and $u_2$ denote, respectively,
the distance of its closest and its farthest point from the $z$-axis.
Let $L$ denote the limit of $u_2/u_1$ as $u_1$ tends to zero. Note that this limit exists, is greater than one,
and does not depend on the radius of $S$, only on the angle of $G_1$ and $G_2$;
this follows from the observation that spherical space `locally' is Euclidean.

Now we choose the radius of the spherical neighborhood of $p$ such that $2(R-1)< L-1$ is satisfied.
Thus, by (1) of Lemma~\ref{lem:limit}, there is a sufficiently small circle $C$ on $S$
touching $G_1$ and $G_2$  such that the truncation of $K$ by the plane containing $C$ has a stable point,
with respect to $o$, on the truncated, disk-shaped part.
Furthermore, by (2) of the same Lemma, we may choose $0<\tau<1$ in a way that
if $C$ is sufficiently small, then the distance of this
point from the boundary of $C$ is at least $\tau$ times the radius of the circle.

Let $H_\varepsilon$ denote the plane perpendicular to and intersecting $[o,p]$,
at the distance $\varepsilon$ from $p$.
Let $\bar{K}_\varepsilon$ denote the truncation of $K$ by this plane.
Let $U_\varepsilon$ be the set of the centers of all the convex bodies $D$ satisfying $\bar{K}_{2\varepsilon} \subseteq D \subseteq K$.
Clearly, we may choose a sufficiently small $\varepsilon$ such that the Morse-Smale complex of $K$ with respect to any point of $U_\varepsilon$ is homeomorphic to that with respect to $o$.
Furthermore, using (3) of Lemma~\ref{lem:truncationbysphere}, we may choose $\varepsilon$ in a way that both the diameter of $U_\varepsilon$, and the diameter of the projection $U^S_\varepsilon$
on $S$ of $U_\varepsilon$ from $q$ is at most $c \varepsilon^2$ for some $c>0$ independent of $\varepsilon$.
Note that for small values of $\varepsilon$, these projections belong to $S$ (recall that $q$ is the center of the sphere of radius $R$ forming the spherical neighborhood of $p$ in $\bd K$).
Furthermore, for any fixed $\delta > 0$, we may also choose $\varepsilon$ in a way that for any $u \in U_{\varepsilon}$,
the intersection points of any integral curve of the gradient flow with respect to $u$, connecting from the projection of $u$ from $q$ on $S$ to a saddle point of $K$, with the circle $K \cap H_\varepsilon$, lies on an arc of angle not greater than $\delta$.
We denote these integral curves by $\Gamma^u_1, \Gamma^u_2, \ldots, \Gamma^u_k = \Gamma^u_0$.

Let $C_\varepsilon$ be the circle that touches $G_1$, $G_2$ and $H_\varepsilon$ on the side of $H_\varepsilon$
not containing $o$,
and $K_\varepsilon$ be the truncation of $K$ by the plane containing $C_\varepsilon$.
Note that the radius of $C_\varepsilon$ is at least $\bar{c} \sqrt{\varepsilon}$
for some positive constant $\bar{c}$.
Thus, for sufficiently small $\varepsilon$, the following are satisfied.
\begin{itemize}
\item for any $u \in U_\varepsilon$, $C_\varepsilon$ intersects exactly those curves from amongst
the $\Gamma^u_i$s that we want to intersect with the planar section.
\item No such intersection point is on $G_1$ and $G_2$.
\item The orthogonal projection of $U_\varepsilon$ on the plane of $C_\varepsilon$ is contained in the interior
of $C_\varepsilon$.
\item $U^S_\varepsilon$ is contained in $K_\varepsilon$, and is disjoint from $C_\varepsilon$.
\end{itemize}

Clearly, under the conditions described in the previous
paragraph, $K_\varepsilon$ has three equilibrium points in the closed half plane bounded by $H_\varepsilon$ and
containing $p$: a stable point $p'$ `near' $p$ (that is, in $U^S_\varepsilon$), another stable point $s$ on the planar disk bounded by $C_\varepsilon$, and a saddle point $h$ on $C_\varepsilon$.
This saddle point is connected to both $p'$ and $s$: the integral curve connecting it to $p'$
is a great circle arc, and the other one is a straight line segment.
Furthermore, for every $\Gamma_i$ on the truncated side of $G_1 \cup G_2$, there is a (piecewise differentiable)
integral curve of $K_\varepsilon$ connecting $s$ to the corresponding saddle point of $K_\varepsilon$,
and for every  $\Gamma_i$ on the other side of $G_1 \cup G_2$, there is a similar curve ending at $p'$.
This implies, by exclusion, that $h$ is connected
to the two unstable points in the two chosen regions containing $G_1$ and $G_2$.

Finally, observe that by replacing the truncating plane containing $C_\varepsilon$ by a ball of sufficiently
large radius results in a homeomorphic Morse-Smale complex. Furthermore, `to smooth out' $\bd K_\varepsilon$,
we may replace a neighborhood of $C_\varepsilon$ by a sufficiently small toroidal surface (a part of the surface of a torus), which
results in a convex body $K'$ with a $C^1$-class boundary (where the equilibria has $C^2$-class neighborhood), such that its
Morse-Smale complex is homeomorphic to the one split in the required way.
The fact that no new equilibria appear can be seen geometrically.
We remark that $K_{2\varepsilon} \subseteq K'$, and thus, the center of $K'$ is still contained in $U_\varepsilon$.
Finally, we may apply Lemma~\ref{lem:smoothing_subroutine} to obtain a convex body with a $C^\infty$-class boundary,
which finishes the proof for a stable point.

For an unstable point we may apply a similar consideration. In this case, instead of a truncation with a ball
of large radius, we expand the original ball with a circular cone the axis of which contains the ball center $q$.

\begin{rem}
One may wonder whether we could have proven our statement using only one local truncation by a plane. However, as we already indicated in the Introduction, this
appears to be impossible:
The phenomenon described in Lemma~\ref{lem:limit} shows that to carry out an arbitrary splitting of an arbitrary  (given) vertex we need one additional free parameter and this is
the radius of the truncating sphere. In terms of local quantities, this implies that even in the special case when the chosen equilibrium is an umbilical point (i.e. its vicinity is spherical),
the required splitting \emph{cannot} be realized by a planar truncation unless the principal curvature at this point is contained in a given interval determined by the directions of the edges starting at this vertex.
As a consequence, we need to adjust the principal curvature \emph{before} truncating with a plane. This step is achieved in Lemma~\ref{lem:smallradius}.
\end{rem}

\section{Summary and related issues}\label{sec:summary}

The central idea of our paper is to associate vertex splittings with
localized geometric transformations. The latter are defined in such a way
that we can control all combinatorial possibilities. Next we show
that vertex splittings arise in a spontaneous way in various
geometric settings where they may or may not exhaust the
full combinatorial catalog. From this point of view, our construction
creates a framework to study these processes.

\subsection{A road map for the generic bifurcations of one-parameter vector fields}

If we consider generic, 1-parameter families of gradient vector fields on $\S^2$ then it is not true that every element of such a family is Morse-Smale.
Rather, these families produce two types of singularities at which this property may fail: saddle-node bifurcations and saddle-saddle connections \cite{Arnold2}.
The former is a local bifurcation while the latter is a non-local bifurcation. A saddle-node bifurcation corresponds to a vertex splitting or a face contraction on the quasi-dual graph representation of the Morse-Smale complex, while a saddle-saddle connection corresponds to a transformation called \emph{diagonal slide} \cite{Negami}.
Each gradient vector field can be uniquely associated with the quasi-dual graph representation of its Morse-Smale complex, so the evolution of one-parameter families can be studied
on a metagraph $\mathcal {G}$ the vertices of which are graphs $Q^2 \in \mathcal{Q}^2,$ representing the Morse-Smale complexes
and the edges of $\mathcal {G}$ correspond to generic bifurcations in one-parameter families.
Any such family will appear
as a path on $\mathcal {G}$.  Convex bodies associated with some selected graphs (selected vertices of $\mathcal G$) are illustrated in Figure \ref{fig:metagraf}/a3.  A small portion of $\mathcal {G}$ is illustrated in Figure \ref{fig:metagraf}/a1.
Solid edges represent saddle-node bifurcations,
dashed edges represent saddle-saddle connections,  Figure \ref{fig:metagraf}/a2 shows the latter inside the primary equilibrium class
$\{S,U\}=\{2,3\}$.
Theorem \ref{thm:main} states that all vertices of $\mathcal{G}$ can be represented with suitably chosen, smooth convex bodies.
We did not prove the existence of convex bodies carrying  gradient fields which exhibit the codimension 1 bifurcations, corresponding to the edges of $\mathcal{G}$.
Nevertheless, as we point out below, our geometric construction suggests that solid edges (saddle-nodes) exist on convex bodies.
Consider two tertiary classes corresponding to two vertices of $\mathcal{G}$ and assume that they
are in adjacent primary classes, i.e. either the number of stable points or the number of unstable points  differs by one (but not both).
Such pairs neighbour tertiary classes can be observed in Figure \ref{fig:metagraf}/a1, for example $(b,c)$, $(c,e)$,$(c,f)$.
It follows from our geometrical construction (Lemma \ref{lem:main}) that if this pair of vertices of $\mathcal{G}$ are connected by an edge (e.g. $(b,c)$, $(c,e)$) then  there is a one-parameter family $K(\lambda)$ of convex bodies 
between those two tertiary classes that except for one arbitrary short interval $\lambda \in (\lambda ^{\star}-\epsilon,\lambda^{\star}+\epsilon)$, all convex bodies $K(\lambda)$ are generic. Based on this we formulate

\begin{con} \label{con:saddlenode}
For any generic, one-parameter family $\mathbf{v}(\lambda)$ of gradient
vector fields on $\S^2$ exhibiting a codimension 1 saddle-node bifurcation
for one single isolated parameter value $\lambda=\lambda ^{\star}$
there exists a one-parameter family $K(\lambda)$ of (not necessarily
smooth) convex bodies with gradient fields $\nabla r_{K(\lambda)}$
such that the latter is topologically equivalent to $\mathbf{v}(\lambda)$
for every value of $\lambda$.
\end{con}

Theorem \ref{thm:main}, together with Conjecture \ref{con:saddlenode} state that the \em oriented \rm subgraph $\mathcal {G}_v\subset\mathcal {G}$ (illustrated in Figure \ref{fig:metagraf}/b1) containing only vertex splittings, exists among gradient fields associated with
convex bodies. Combinatorial expansion sequences (\ref{gen}) associated with an $n$-vertex graph $Q_n \in \mathcal {Q}_n$ appear on this oriented metagraph as an oriented path of length $n-2$, starting at the root ($P_2$). Observe that in the $k$th step a vertex in the box-diagonal $S+U=k$ is selected. Two such sequences are illustrated in Figure \ref{fig:metagraf}/b2. As pointed out above, beyond saddle-nodes, one-parameter gradient fields also undergo generic, codimension 1 saddle-saddle
bifurcations which are non-local, however, it is not known whether there exists a geometric correspondence for this combinatorial connection inside the primary equilibrium classes.
Although our current geometric argument does not provide any direct hint for their existence on convex bodies, nevertheless we formulate

\begin{con} \label{con:saddlesaddle}
For any generic, one-parameter family $\mathbf{v}(\lambda)$ of gradient
vector fields on $\S^2$ exhibiting any codimension 1 bifurcation
for one single isolated parameter value $\lambda=\lambda ^{\star}$
there exists a one-parameter family $K(\lambda)$ of (not necessarily
smooth) convex bodies with gradient fields $\nabla r_{K(\lambda)}$
such that the latter is topologically equivalent to $\mathbf{v}(\lambda)$
for every value of $\lambda$.
\end{con}

stating that all edges of the metagraph $\mathcal G$ can be realized among convex bodies.
Beyond theoretical interest, these metagraphs admit the study of interesting physical phenomena some of which we briefly discuss below.

\begin{figure}[htp]
\includegraphics[width=1.1\textwidth]{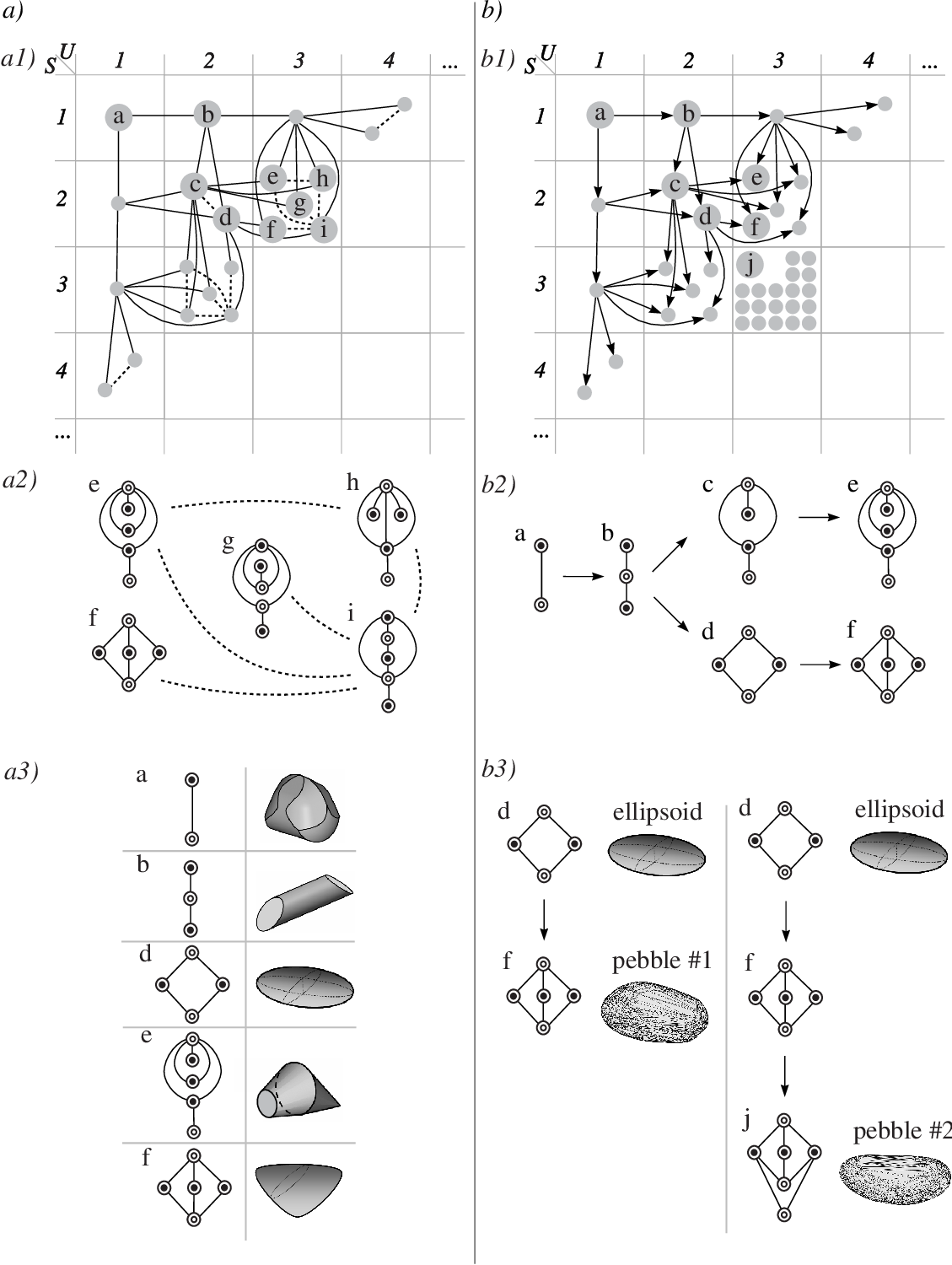}
\caption[] {a1) Metagraph $\mathcal {G}$ corresponding to generic bifurcations in one-parameter families of vector fields on the sphere
a2) Saddle-saddle connections inside the primary equilibrium class $\{S,U\}=\{2,3\}$
a3) Examples for convex bodies with some selected graphs
b1) Oriented subgraph $\mathcal {G}_v\subset\mathcal {G}$ corresponding to vertex splittings
b2) Two selected combinatorial expansion sequences
b3) Morse-Smale complexes associated with real pebbles, derived with expansion sequences from the ellipsoid}
\label{fig:metagraf}
\end{figure}\subsection{Inhomogeneous bodies}

Our first example is inhomogeneous bodies. So far, throughout the paper we assumed convexity and material homogeneity. Relaxing the latter constraint is equivalent to keep the convex surface $\bd K$ as the boundary of the body but let the mass be concentrated at the center of mass $C$. As the location of $C$ is varied in time as a curve $\mathbf{r}_C(t)$, it generates a one-parameter family of gradient vector fields on $\bd K$.
A classical result in catastrophe theory states that the number of critical points of the gradient changes if and only if $\mathbf{r}_C(t)$ transversely passes through one of the two caustics of the body \cite{Poston}. Caustics (also known as focal surfaces) are the two surfaces formed by the curvature centers corresponding to the principal curvatures of $\bd K$. Figure~\ref{fig:caustic} a)-c) shows the two caustics of an ellipsoid, a)  corresponding to curvature minimum, b) corresponding to curvature maximum, c) intersection (superposition) of both caustics.
When $\mathbf{r}_C(t)$ transversely crosses the caustic defined by the minimal principal curvature, a saddle and an unstable point meet at a saddle-node; when $\mathbf{r}_C(t)$ transversely crosses the other caustic, a saddle and a stable point collide. Every saddle-node bifurcation corresponds to a vertex splitting (or face contraction, depending on the direction) on the quasi-dual Morse-Smale graph, so at each such event
the corresponding path on $\mathcal G$ will move from one  box-diagonal $S+U=k$ to one of its neighbor diagonals. Figure~\ref{fig:caustic} d) shows the different Morse-Smale graphs in the different domains determined by the intersections of the two caustics.
It is easy to see that if $C$ is located far enough from the center of mass of the homogeneous body
then the corresponding Morse-Smale complex is represented by the path graph $P_2$.

\begin{figure}[ht]
\includegraphics[width=\textwidth]{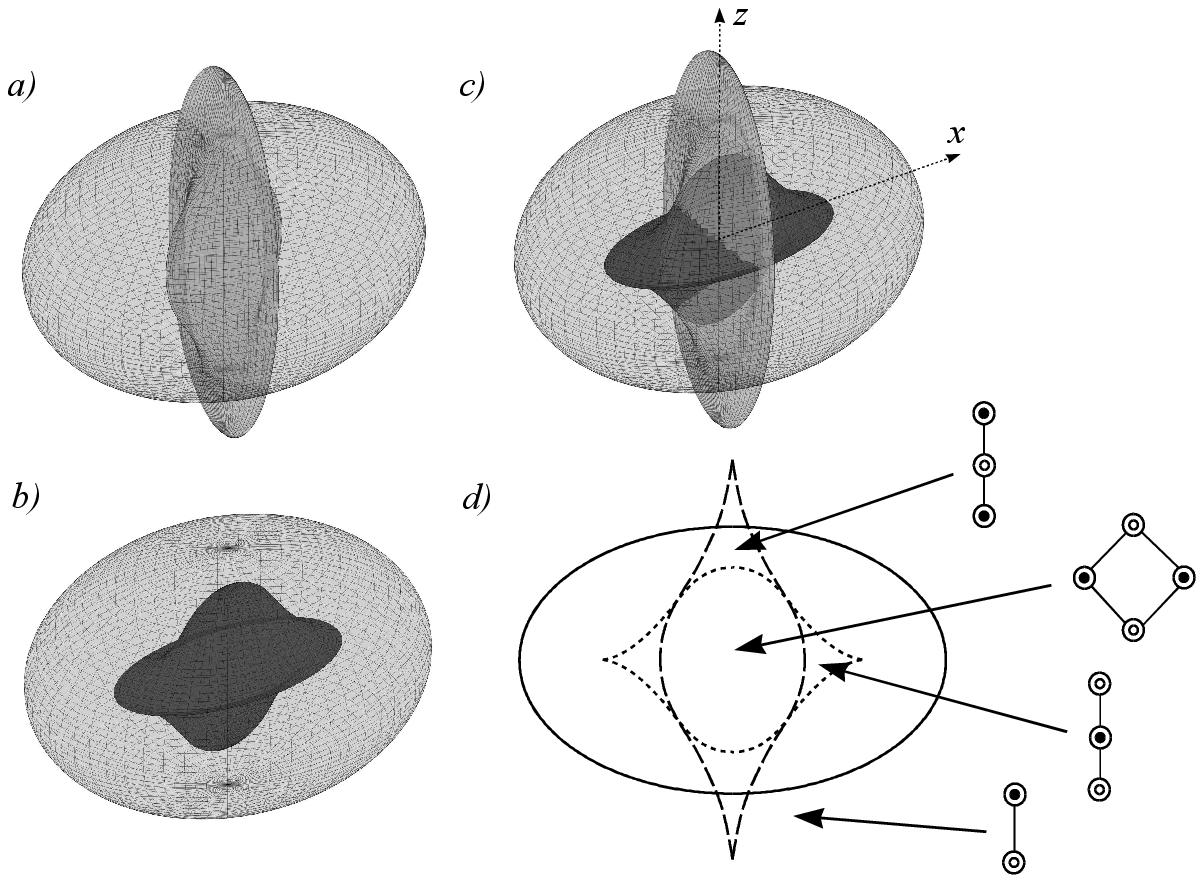}
\caption[]{Caustics of an ellipsoid.
a) Caustic defined by the minimal principal curvature
b) Caustic defined by the maximal principal curvature
c) Superposition  of the two caustics
d) Quasi-dual Morse-Smale graphs (in class ${\mathcal Q}$) in the different domains, shown on the plane section through axes $x$ and $z$}
\label{fig:caustic}
\end{figure}

\subsection{Collisional abrasion: chipping of rocks}

Our second example is pebble abrasion via collisions. This process is most often described by averaged geometric PDEs, the most general such model is given by Bloore \cite{Bloore} as

\ben
v=1+2bH+cK
\een
where $v$ is the attrition speed along the inward surface normal, $H,K$ are the Mean and Gaussian curvatures, $b,c$ are constants.
Solutions of these PDEs correspond to (inward propagating) wave fronts.
The  actual physical process is somewhat different: it is based on discrete collisions where small amounts
of material are being removed in a strongly localized area.
Simple but natural interpretations of the discrete, physical abrasion process are \em chipping algorithms \rm \cite{DSV},\cite{SDWH} and \cite{Krapivsky} where in each step a small amount of material is chipped off at point $p$ by intersecting the body with a plane resulting as a small parallel translation of the tangent at $p$.
We call such an operation a \em chipping event \rm and their sequence a \em chipping sequence. \rm

In Section \ref{sec:geometry} we showed that any vertex splitting can be achieved by a suitably chosen convexity-preserving local truncation.
It is not very difficult to show a related, though converse statement: any \emph{sufficiently small} chipping event will either leave the Morse-Smale complex invariant or result in one or two consecutive vertex splittings.
If we regard the material abraded in a chipping event as a random variable $\chi$
with very small, but finite expected value $E(\chi)$ and very small variance (as it is often done in chipping algorithms) then we expect that for some time intervals chipping events will be sufficiently small
to form finite chipping sequences the \em subsequences \rm of which are geometric expansion sequences of type (\ref{gen1}).

Chipping sequences do not represent a rigorous, algorithmic discretization of the PDEs, rather, they
can be regarded as an alternative, discrete approximation of the physical process. As it was pointed out in \cite{DLS} and \cite{Domokos}, pebble surfaces display equilibria on two, well-separated scales.
While the PDE description accounts for the evolution of global equilibria, local equilibria, corresponding to the fine structure of the surface are only captured by the chipping model.
Figure \ref{fig:pebble} shows the high-accuracy scan of a real pebble with equilibria and the primal representation of the Morse-Smale complex. We can observe flocks of
local equilibria accumulated around global equilibrium points.
In fact, one motivation behind chipping algorithms is to better understand the interplay between the two scales. Chipping models appear to be successful in explaining laboratory experiments (cf. \cite{Krapivsky} and \cite{SDWH})
as well as geological field observations \cite{SDWH}. The connection between chipping models and our geometric expansion algorithms suggests that the number of \emph{local} equilibrium points may increase in abrasion processes
for finite time intervals. Whether and how this process interacts with the evolution of \emph{global} equilibria is an open question not addressed in the current paper.

Figure \ref{fig:metagraf}/b3 illustrates two, rather short geometrical expansion sequences leading to Morse-Smale complexes associated with real pebbles. While the abrasion of these
pebbles was not monitored, given their simple Morse-Smale complex and nearly-ellipsoidal shape it is realistic to assume that these geometric expansion sequences
are  subsequences of the actual physical abrasion process (modeled by chipping sequences) which produced these shapes. Needless to say, many more experiments are needed to verify this theory.

\begin{figure}[ht]
\includegraphics[width=300pt]{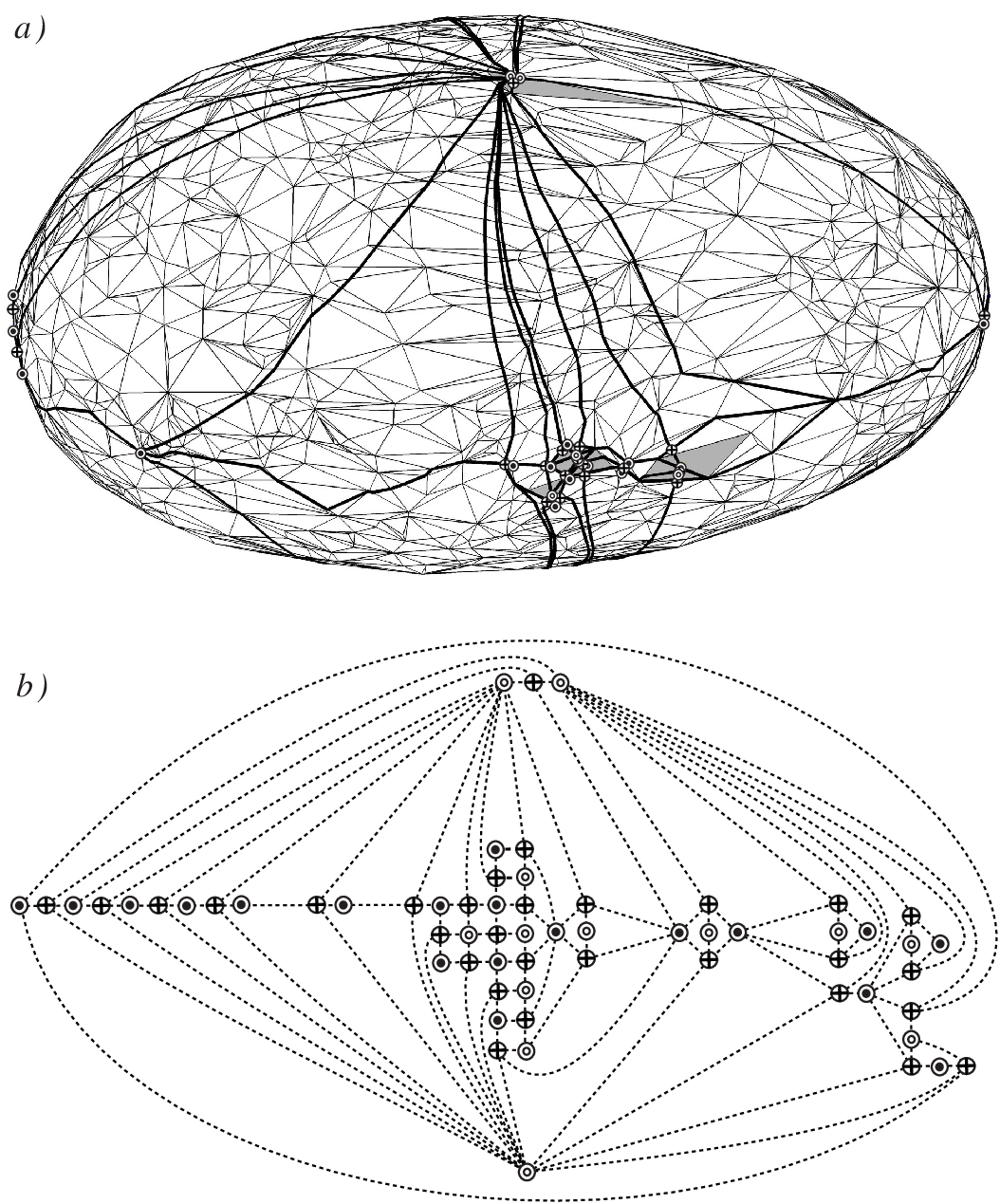}
\caption[]{The Morse-Smale complex of a real pebble: primal Morse-Smale graph in class ${\mathcal Q}^{3\star}$
a) drawn on the pebble's surface
b) drawn on the plane }
\label{fig:pebble}
\end{figure}

\subsection{Concluding remarks}

In this paper we showed that Morse-Smale functions $\mathcal{M}_C$ associated with convex bodies exhaust all combinatorial possibilities, i.e., to any two-colored quadrangulation $\bar Q^2 \in \mathcal{Q}^2$ on $\S^2$ there
exists a convex body $\bar K$ such that $\bar Q$ is the quasi-dual representation of the Morse-Smale complex associated with $\bar K$. We proved our claim by showing that to each possible
combinatorial expansion sequence $Q^2_i$ generated by vertex splittings, there exists a coupled geometric expansion sequence $K_i$ such that
$Q^2_i$ is the quasi-dual representation of the Morse-Smale complex associated with $K_i$. Geometric expansion sequences were created by local, convexity preserving truncations of the
convex body. Beyond proving the existence of all combinatorially possible
convex shapes, we also generalized the classification scheme of \cite{Varkonyi}.  We also showed that geometric expansion sequences appear to be part of the natural geometrical description of collisional abrasion.

\section{Acknowledgement}
This research was supported by OTKA grant T104601. The authors thank an anonymous referee for suggesting substantial improvements to the paper.
The authors are indebted to E. Makai Jun., G. Etesi and Sz. Szab\'o for their valuable comments on smooth approximations of continuous functions.
Z. L\'angi also acknowledges the support of the Fields Institute for Research in Mathematical Sciences, University of Toronto, Toronto ON, Canada,
and the J\'anos Bolyai Research Scholarship of the Hungarian Academy of Sciences.


\begin{thebibliography}{99}

\bibitem{Andreev} E.M. Andreev, Convex polyhedra in Lobachevsky spaces, Math. Sb. (N.S.) 81(123) (1970), 445-478.

\bibitem{Archdeacon} Archdeacon, D., Hutchinson, J., Nakamoto, A., Negami, S. and Ota, K.,
	Chromatic numbers of quadrangulations on closed surfaces, J. Graph Theory 37 (2001), 100-114.

\bibitem{Arnold} Arnold, V.I.,
	Ordinary differential equations, 10th printing, MIT Press, Cambridge, 1998.
	
\bibitem{Arnold2} Arnold, V.I. (Ed.),
	Dynamical Systems V: Bifurcation Theory and Catastrophe Theory, Springer-Verlag, Berlin, 1994.
		
\bibitem{Bagatelj} Bagatelj, V.,
	An inductive definition of the class of 3-connected quadrangulations of the plane, Discrete Math. {\bf 78 }(1989), 45-53.
	
\bibitem{Bauer} Bauer, U., Lange, C. and Wardetzky, M.,
	Optimal topological simplification of discrete functions on surfaces, Discrete Comput. Geom. 47 (2012), 347-377.

\bibitem{BF87} Bonnesen, T., Fenchel, W.,
Theory of Convex Bodies, Moscow, Idaho: L. Boron, C. Christenson and B. Smith, BCS Associates, 1987.

\bibitem{Bloore} Bloore, F.J.,
The Shape of Pebbles, Math. Geol. 9 (1977) 113-122.

\bibitem{Bremer} Bremer, P.T., Edelsbrunner, H., Hamann, B. and Pascucci, V.,
	A multi-resolution data structure for two-dimensional Morse-Smale functions, Proceeding VIS '03 (2003), 139-146.
	
\bibitem{Brinkmann} Brinkmann, G., Greenberg, S., Greenhill, C., McKay, B.D., Thomas, R. and Wollan, P.,
	Generation of simple quadrangulations of the sphere, Discrete Math. 305 (2005), 33-54.
	
\bibitem{McKay} Brinkmann, G. and McKay, B.D.,
	Fast generation of planar graphs, MATCH Commun. Math. Comput. Chem  58 (2007), 323-357.
	
\bibitem{Conway} Conway, J.H. and Guy, R.,
	Stability of polyhedra, SIAM Rev. 11 (1969), 78-82.
	
\bibitem{Dawson} Dawson, R.,
	Monostatic Simplexes, Amer. Math. Monthly 92 (1985), 541-546.F
	
\bibitem{DawsonFinbow} Dawson, R. and Finbow, W.,
	What shape is a loaded die?, Math. Intelligencer 22 (1999), 32-37.
	
\bibitem{Dey} Dey, T.K., Li, K., Luo, C., Ranjan, P., Safa, I. and Wang, Y.,
Persistent Heat Signature for Pose-oblivious Matching of Incomplete Models, Computer Graphics Forum 29 (2010), 1545-1554.
	
\bibitem{Diestel} Diestel, R.,
	Graph Theory, 3rd edition, Springer-Verlag, Heidelberg, 2005.
	
\bibitem{DLS} Domokos G., L\'angi Z. and Szab\'o, T.,
 	On the equilibria of finely discretized curves and surfaces, Monatsh. Math. 168 (2012) 321-345.

\bibitem{Domokos} Domokos G., Sipos A.\'A. and Szab\'o, T.,
 	The mechanics of rocking stones: equilibria on separated scales, Math. Geosci. 44 (2012), 71-89.

\bibitem{DSV} Domokos G., Sipos A.\'A. and  V\'arkonyi P.
 Continuous and discrete models for abrasion processes, Per. Pol. Arch.  40 (2009) 3-8.

\bibitem{Dong} Dong, S., Bremer, P.-T., Garland, M., Pascucci, V. and Hart, J.C.,
	Spectral surface quadrangulation, ACM T. Graphic 25 (2006), 1057-1066.
		
\bibitem{Edelsbrunner} Edelsbrunner, H., Harer, J. and Zomorodian, A.,
	Hierarchical Morse-Smale complexes for piecewise linear 2-manifolds, Discrete Comput. Geom. 30 (2003), 87-107.

\bibitem{E98} Evans, L.,
Partial differential equations, Graduate Texts in Mathematics 19, American Mathematical Society, Providence RI, 1998.
	
\bibitem{Fusy} Fusy, E.,
	Counting unrooted maps using tree-decomposition, Seminaire Lotharingien de Combinatoire 54A (2007), Article B54Al
	
\bibitem{G02} Ghomi, M.,
 The problem of optimal smoothing for convex functions, Proc. Amer. Math. Soc. 130 (2002), 2255-2259.
	
\bibitem{Gross} Gross, J. T. and Yellen, J.,
	Graph Theory and Its Applications, 2nd ed., Boca Raton, FL: CRC Press, 2006.	
	
\bibitem{Gyulassy} Gyulassy, A., Natarajan, V., Pascucci and V., Hamann, B.,
	Efficient computation of Morse-Smale complexes for three-dimensional scalar functions, IEEE Trans. Vis. Comput. Graph. 13 (2007), 1440–1447.
	
\bibitem{Archimedes1} Heath, T. I. (Ed.),
The Works of Archimedes, Cambridge University Press, 1897.
	
\bibitem{Heppes} Heppes, A.,
	A double-tipping tetrahedron, SIAM Rev.  9 (1967), 599-600.

\bibitem{H76} Hirsch, M.,
Differential topology, Graduate Texts in Mathematics 33, Springer-Verlag,
New York-Heidelberg, 1976.

\bibitem{Kapolnai} K\'apolnai, R. and Domokos, G.,
	Inductive generation of convex bodies, in: The 7th Hungarian-Japanese Symposium on Discrete Mathematics and Its Applications, 2011, pp. 170-178.

\bibitem{Krapivsky} Krapivsky, P.L. and Redner S.,
	Smoothing a rock by chipping, Phys. Rev. E 9 (2007), 75(3 Pt 1):031119.

%\bibitem{Milnor} Milnor, J.,
%	Morse Theory, Princeton Univ. Press, New Jersey, 1963.

%\bibitem{Nakamoto} Nakamoto, A.,
%	Generating quadrangulations of surfaces with minimum degree at least 3, J. Graph Theory  30 (1999), 223-234.
	
\bibitem{Negami} Negami, S. and Nakamoto, A.,
	Diagonal transformations of graphs on closed surfaces, Sci. Rep. Yokohama Nat. Univ., Sec. I  40 (1993), 71-97.

\bibitem{Poston} Poston, T. and Stewart, J.,	
	Catastrophe theory and its applications, Pitman, London, 1978.	

\bibitem{RHD07} R.K.W. Roeder, J.H. Hubbard and W.D. Dunbar, Andreev's theorem on hyperbolic polyhedra. Ann. Inst. Fourier (Grenoble), 57(3) (2007), 825-882.
	
\bibitem{SDWH} Sipos A.\'A., Domokos G., Wilson A. and Hovius N.
A Discrete Random Model Describing Bedrock Erosion, Math. Geosci. 43  (2011) 583-591.
		
\bibitem{Varkonyi} V\'arkonyi, P.L. and Domokos, G.,
	Static equilibria of rigid bodies: dice, pebbles and the Poincar\'e-Hopf Theorem, J. Nonlinear Sci. 16 (2006), 255-281.
		
\bibitem{Zamfirescu} Zamfirescu, T.,
How do convex bodies sit?, Mathematica  42 (1995), 179-181.

\bibitem{Zomorodian} Zomorodian, A.,
	Topology for computing, Cambridge University Press, 2005.	





\end{thebibliography}
\end{document}